\RequirePackage{fix-cm}
\documentclass{svjour3}                     
\smartqed  
\usepackage{graphicx}

\usepackage{layout,yfonts, graphicx, amsmath, amssymb,euscript, enumerate, enumitem, bbold, bbm, mathrsfs,mathtools}
\usepackage{color}
\usepackage{xurl}

\numberwithin{equation}{section}

\DeclareMathAlphabet{\mathpzc}{OT1}{pzc}{m}{it}

\DeclareMathOperator{\Lp}{L}
\DeclareMathOperator{\trans}{T}
 
\DeclareMathOperator{\loc}{loc}
\DeclareMathOperator{\hol}{C}

\DeclareMathOperator{\sob}{W}

\DeclareMathOperator{\expo}{e}

\newcommand{\tmn}{\mathpzc{n}}
\newcommand{\tmp}{\mathpzc{p}}

\newcommand{\blue}{\textcolor[rgb]{0.00,0.00,1.00}}
\newcommand{\E}{\mathbbm{E}}

\newcommand{\set}{\mathcal{O}}

\newcommand{\R}{\mathbbm{R}}

\newcommand{\uno}{\mathbbm{1}}
\newcommand{\der}{\mathrm{d}}
\newcommand{\F}{\mathbbm{F}}
\newcommand{\Pro}{\mathbbm{P}}
\newcommand{\BE}{\begin{equation}}
	
	\newcommand{\EE}{\end{equation}}
\newcommand {\BA}{\begin{align}}
	\newcommand{\EA}{\end{align}}
\newcommand{\eqdef}{\raisebox{0.4pt}{\ensuremath{:}}\hspace*{-1mm}=}

\DeclareMathOperator{\Cof}{Cof}

\newcommand{\parcial}[2]{\frac{\partial #1}{\partial #2}}

\newtheorem{prop}{Proposition}[section]{\bfseries}{\itshape}
\spnewtheorem{theor}[prop]{Theorem}{\bfseries}{\itshape}
\newtheorem{corol}[prop]{Corollary}{\bfseries}{\itshape}
\newtheorem{rem}[prop]{Remark}{\bfseries}{\itshape}
\newtheorem{lema}[prop]{Lemma}{\bfseries}{\itshape}

\journalname{AMOP}
\allowdisplaybreaks

\begin{document}

\title{Optimal extraction   with an impact on  {diffusion-jump} pricing
}


\author{Johanna Garz\'on       \and
        Jhonatan S. Mora Rodr\'iguez \and
        Harold A. Moreno-Franco
}


\institute{J. Garz\'on \at
              Departamento de Matem\'aticas, Facultad de Ciencias, Universidad Nacional de Colombia, Ave Cra 30 No. 45-3, Bogot\'a, Colombia \\
              \email{mjgarzonm@unal.edu.co}           
           \and
           J.S. Mora Rodr\'iguez \at
              Departamento de Matem\'aticas, Facultad de Ciencias, Universidad Nacional de Colombia, Ave Cra 30 No. 45-3, Bogot\'a, Colombia\\
              \email{jsmoraro@unal.edu.co}
             \and
             H. A. Moreno-Franco \at
             Department of Statistics and Data Analysis, Laboratory of Stochastic Analysis and its Applications,  National Research University Higher School of Economics, 11 Pokrovsky Bulvar, Moscow, Federation of  Russia.\\
             \email{hmoreno@hse.ru}
}

\date{Received: date / Accepted: date}

\maketitle

\begin{abstract}
We study an optimal extraction problem where the agent's actions in the spot market exert an additive proportional  negative impact on the commodity price.  The commodity price dynamics,  prior to any activity by the agent, are evolved by a drifted Brownian motion with jumps.  The agent's primary aim is to identify an optimal extraction strategy that maximizes their expected net profits.
\end{abstract}

\section{Introduction}

In recent years, certain commodities in the spot market have encountered negative prices,  leading to significant challenges for trading sectors. A notable example is the historic occurrence of negative oil prices for the US oil benchmark, West Texas Intermediate (WTI). On April $20^{\text{th}}$ 2020,  the price dropped dramatically from $\$17.85$ at the beginning of the trading day to a staggering $-\$37.63$ at the close, marking the largest one-day drop for US crude  {oil} in history.  As the expiration date for the May contract approached, traders faced the challenge of acquiring physical barrels of oil. This difficulty was compounded by storage facilities reaching their maximum capacity. Concurrently,  the June contract experienced increased trading volumes, rising to \$21 and creating a super-contango futures market. The large spread between the two contracts caused traders to refrain from rolling, holding, or accepting delivery. Consequently, the May contracts were sold off, pushing  {the price} down  even further and increasing bearish market sentiment. For more detailed information, refer to \cite{GRI2020}.

Another case where the commodity price attained negative  values and producers were forced to sell it for nothing, or less than nothing, is the natural gas of Alberta (Canada) on October $5^{\text{th}}$ and $9^{\text{th}}$, and September $25^{\text{th}}$, 2017.  This unfortunate scenario was repeated  on May $4^{\text{th}}$, 2018.  For more details, see \cite{F2017,US2018}.

Following the dip in the US benchmark oil price below zero, it has become standard practice for trading parties to include clauses in their contracts that necessitate negotiations to modify the agreement terms in case of price crashes or other market disruptions, including regulatory changes. For more information, refer to \cite{HD2021}. These clauses come in two primary types:
\begin{itemize}
	\item The first establishes a zero-price floor, meaning the seller will not be required to pay the buyer if the price is negative. Alternatively, the seller may offer a predetermined lump sum price floor to provide nominal consideration to the buyer.
	
	\item The second type of clause stipulates that the seller must pay the buyer to accept the commodities if the price is negative.
\end{itemize}
Additionally, on October $7^{\text{th}}$, 2020, the International Swaps and Derivatives Association (ISDA) established similar model clauses, which the parties could integrate into their commodity trading documentation; see \cite{I2020,S2020}.

The behaviour of commodity prices in the spot market has led several authors to propose alternative stochastic price models, diverging from the log-price model, to better capture the negative aspects of commodity prices.  For instance, the Ornstein-Uhlenbeck process has been applied to model the price of various commodities. Chaiyapo and Phewchean \cite{CP2017} applied this process to model the Thai commodity market, while Trasierra and Carrasco \cite{TC2016} used it to model gold spot prices for evaluating a mine investment project using the real options approach in Peru. Additionally, the Ornstein–Uhlenbeck process has also been employed in the  {electric power} sector, with Lucia and Schwartz \cite{LS2002} utilizing it to model electricity prices. While the literature justifies the use of stochastic processes with jumps to model commodity prices  (see, for instance,  \cite{BG2020,HLY2012}), as far as we are aware, there is currently no research that has implemented alternative models, apart from the log-price, using stochastic processes with jumps.

Therefore, considering the various real-life scenarios described earlier and the limited information about alternative price models of commodities that involve stochastic processes with jumps beyond the log-price model, the primary aim of the paper is to analyse an optimal extraction problem with an impact on the price, where  the dynamics of the commodity price, before any agent activity, follow  a drifted Brownian motion with jumps, as shown in \eqref{eq1}. In this context, an agent in the spot market seeks to sell a substantial initial quantity of a commodity at any time, with decisions made having an additive proportional negative impact on the commodity price. Furthermore, the agent's decisions align with the first clause mentioned earlier, where there is no obligation to sell when the commodity price is negative or zero. Additionally,  any action taken by the agent does not affect the commodity price until it reaches either  negative or zero values. The   main goal of the agent is to find an optimal extraction rule that maximizes their expected net profits, as illustrated in \eqref{fV}.  

Since 1931, when Hotelling first addressed the topic of optimal extraction (without considering any impact on the price) in his seminal paper \cite{H1931}, many authors have delved into this area. In his work, he introduced a deterministic price model to explore an optimal extraction problem. Building upon similar assumptions regarding commodity prices, subsequent research on optimal extraction was conducted by Hanson \cite{H1980}, Lin and Wagner \cite{LW2007}, Pindyck \cite{P1978}, as well as Solow and Wan \cite{SW1976}. Furthermore, Pindyck \cite{P1980}, Ferrari and Yang \cite{FY2018}, Pemy \cite{P2022}, and Zhang et. al. \cite{ZZZ2023} have also investigated variations of the optimal extraction problem. Specifically, Pindyck \cite{P1980} modeled commodity prices as a geometric Brownian motion, Ferrari and Yang \cite{FY2018} utilized a Brownian motion with regime switching, Pemy \cite{P2022} assumed the log-price of the commodity follows an Ornstein-Uhlenbeck process with jumps and regime switching, and Zhang et al. \cite{ZZZ2023} considered in their model that the price follows a geometric Brownian motion with regime switching.

More recently, Ferrari and Kock \cite{FK2021} have investigated an optimal extraction problem with an additive proportional negative impact, considering the dynamics of the commodity price are governed by an Ornstein-Uhlenbeck process prior to any agent activity. Similarly,  Kock addressed  this problem in his  PhD thesis \cite{K2020}, taking  the underlying process as  a standard Brownian motion with drift. In a similar vein, literature addressing this topic explores scenarios where the agents' decisions exert a multiplicative influence on prices; for instance, refer to \cite{AZ2017,GZ2015,HMP2019}.

We would like to highlight that our problem falls within the framework of singular stochastic control. Specifically, the first component of our two-dimensional controlled process is characterized by a controlled Brownian motion that includes both drift and jumps, effectively modelling the commodity price at time $t$. The second component represents the quantity of commodities held by the agent at that same moment. Both components are influenced by the agent's decisions at time $t$, which are governed by a singular stochastic control strategy.

For those are interested in related literature, we recommend consulting \cite{KM2018,M2018,Z1992}. However, due to the unique aspects of our problem, we explicitly derive the optimal strategy and the value function. This contrasts with the previously mentioned literature, where the authors prove the existence and uniqueness of solutions to the Hamilton-Jacobi-Bellman (HJB) equation associated with  their  value function. They subsequently validate these solutions through a verification result indicating they coincide with the value function. Moreover, in their approaches, the existence of an optimal strategy is not always explicitly discussed.

To find an optimal extraction strategy, we will first establish the dynamic variational inequality that governs the value function for our problem, as shown in \eqref{eqHJB}.  Subsequently, by a verification theorem (see  Proposition \ref{teo1}), we will derive necessary conditions on the solution  to the HJB equation,  to determine the optimal extraction strategy.   Our analysis  will then focus on implementing straightforward barrier strategies, where  selling/extraction decisions are based on  comparing the remaining number of commodities with a certain proportion of the difference between a fixed level ${b^{\star}}$ and the commodity price. By solving the HJB equation associated with our problem, and employing  `smooth fit' arguments,  the optimal extraction strategy will be explicitly defined throughout a barrier strategy. This strategy can be outlined as follows: if the commodity price is below a level $b^{\star}$, it is optimal not to sell any commodity. Moreover, if the commodity price exceeds the level $b^{\star}$, it is optimal either to sell off  all available commodities or vend a portion of them, resulting in a reduction in the commodity price.

The remainder of this document is structured as follows. In Section \ref{SectionFP}, we outline the optimal extraction problem. In Section \ref{sec3}, we deduce the HJB equation \eqref{eqHJB} linked with the value function defined in \eqref{fV}. Next, in Subsection \ref{subsectionVT}, we establish a verification theorem; refer to Proposition \ref{PFV}. Moving on to Section \ref{S4}, we introduce the barrier strategies and explicitly determine their expected net profits, as shown in \eqref{v1.1}. The latter satisfies the HJB equation, as shown in Theorem \ref{teo1}, which is detailed and proven in Section \ref{S5}. {Section \ref{stop1} explores the connection to an optimal stopping problem. A sensitivity analysis is conducted in Section \ref{comp1}. Section \ref{S8} summarizes the main contributions
	of our paper and discusses directions for further work.}   Lastly, the justification of some key outcomes overlooked in the text and other supplementary results are provided in Appendix.

\section{Formulation of the problem}\label{SectionFP}

We consider an agent in the spot market whose sole objective is to sell  $y>0$ initial large units of a commodity  at any time $t\geq0$. Each selling interaction of the agent  in the market will be reflected as an additive proportional negative impact on the commodity's spot  price.  When the agent refrains from participating in  transactions within  the market, the price of the commodity $X^{0}=\{X^{0}_{t}:t\geq0\}$ is modelled by the following drifted Brownian motion with  jumps 
\begin{equation}\label{eq1}
	X_{t}^{0}=x+\mu t+\sigma W_t-S^{\tmn}_{t}+S^{\tmp}_{t} \quad \text{for}\ t\ge 0,
\end{equation}
where $x\in\R$ is the initial  commodity price, $\mu\in\R$ and $\sigma>0$. Here $W=\{W_{t}:t\geq0\}$ is a Brownian motion, and $S^{u}=\{S^{u}_{t}:t\geq0\}$, with $u\in\{\tmp,\tmn\}$, is given by   $S_{t}^{u}\eqdef\sum_{i=1}^{N_{t}^{u}}Z_{i}^{u}\quad\text{for}\ t\geq0$,
where $N^{\tmn}=\{N^{\tmn}_t: t\ge 0\}$, $N^{\tmp}=\{N^{\tmp}_t: t\ge 0\}$  are independent Poisson processes  with rates  $\lambda_{\tmn}$ and $\lambda_{\tmp}$, respectively.  Here, $\{Z^{\tmn}_{i}\}_{i\geq 1}$, $\{Z^{\tmp}_{i}\}_{i\geq 1}$ are sequences of i.i.d. non-negative random variables which are also independent of $N^{u}$, with $u\in\{\tmn,\tmp\}$.   In order that \eqref{eq1} is well-defined, consider that  $W,\ S^{\tmn},\, S^{\tmp}$ are defined on a filtered  probability space $(\Omega, \mathcal{F}, \mathbbm{F}, \Pro)$, where $\mathbbm{F}=\{\mathcal{F}_{t}\}_{t\geq 0}$ is the right-continuous complete filtration generated by $(W,S^{\tmn},S^{\tmp})$, and the processes $W$, $S^{\tmn}$ and $S^{\tmp}$ are independent.  

Considering  $\mathcal{T}^{u}\eqdef\{T^{u}_{i}\}_{i\geq0}$, with $u\in\{\tmn,\tmp\}$,  as the  arrival times sequence of $N^{u}$, we have that $Z^{u}_{k}$ indicates the jump in the commodity price occurring at time $T^{u}_{k}$ due to extreme shocks like economic, financial, and geopolitical events; for further details, refer to  \cite{BG2020}.  

To capture the asymmetric leptokurtic characteristics observed in commodity pricing, throughout the paper, it will be assumed  that  $Z^{u}_{k}$  follows a  hyper-exponential distribution $P_{u}$, i.e.,
\begin{equation}\label{cond3.0}
	\frac{{ \der}P_{u}(z)}{{ \der}z}=\sum_{k=1}^{m_{u}}\omega_{k}^{u}\beta_{k}^{u}\expo^{-\beta_{k}^{u}z},  \text{ for } z>0.
\end{equation}
Here $m_{u}$ is a positive integer,  $\beta_{k}^{u}>0$ and $\omega_{k}^{u}>0$ for $1\leq k\leq m_{u}$,  satisfying
\begin{equation}\label{cond3}
	0<\beta_{1}^{u}<\beta_{2}^{u}<\cdots<\beta_{m_u}^{u}\quad \text{and}\quad\sum_{k=1}^{m_u}\omega_{k}^{u}=1.
\end{equation}
This kind of assumption has been employed to model jumps in the log-price of assets; see for example \cite{CK2011,RZ2007}.

As previously mentioned in this section, the agent in the spot market possesses a substantial initial inventory $y>0$, which they intend to sell at any given time $t$. Consequently, the accumulated quantity of extracted and sold commodity by time $t$ is represented by an increasing c\`adl\`ag process $\xi=\{\xi_t:t\geq 0\}$ that is $\mathbbm{F}$-adapted, satisfying $\xi_{t}\leq y$ a.s. for $t\geq0$, and $\xi_{0}=0$. Subsequently, the agent's inventory of commodity at time $t$ is described by
\begin{equation}\label{e1}
	Y_t=y-\xi_t\geq 0.
\end{equation}

Since the agent's decisions to extract and sell impact the commodity price negatively, if at time $t>0$ the agent opts to sell a small quantity $\varepsilon>0$, the price of the commodity decreases by a factor $\alpha>0$, expressed as $\Delta X_{t}=X_{t}-X_{t-}=-\alpha \varepsilon-\Delta S^{\tmn}_{t}+\Delta S^{\tmp}_{t}$.
Furthermore, when a significant extraction and sale of commodity $\Delta \xi_t$ occurs, this action can be perceived as a large number $N$  of individual extractions and sales of small size $\varepsilon=\Delta\xi_t/N$ carried out simultaneously. In such instances, the commodity price at time $t$ follows
\begin{equation}\label{e2}
	X_t=X_{t-}-\alpha N\varepsilon-\Delta S^{\tmn}_{t}+\Delta S^{\tmp}_{t}.
\end{equation} 

Based on  the aforementioned considerations, the spot price of the commodity is determined by
\begin{equation}\label{e3}
	X_{t}=x+\mu t+\sigma W_t-S^{\tmn}_{t}+S^{\tmp}_{t}-\alpha \int_{0}^{t}\der \xi_{s}\quad \text{for}\ t\geq0.
\end{equation}

Since there is no obligation to extract and sell when the commodity price is either negative or zero, and any action taken by  the  agent   does not affect the commodity price until it reaches negative values, throughout of the paper, we will consider strategies $\xi$ that satisfy the following
\begin{equation}\label{cont.1}
	\begin{cases}
		\xi\ \text{is adapted to the filtration}\  \F;\\
		\xi_{0}=0\ \text{and}\ \xi_{t}\ \text{is non-decreasing and}\\
		 \text{is right continuous with left hand limits,}\ t\geq0;\\
		\xi_{t}\leq y\ \text{a.s.}\ t\geq0;\ \der \xi_{t}\equiv0\   \text{ if}\ X_{t}\leq0;\\ \text{and }\ \Delta\xi_{t}\eqdef\xi_{t}-\xi_{t-}\leq \frac{X_{t-}}{\alpha}\ \text{if}\ X_{t-}>0.
	\end{cases}
\end{equation}
It will be stated that $\xi$ is  admissible if it satisfies \eqref{cont.1}. The collection of all admissible strategies is denoted by $\mathcal{A}(y)$. 

If $\xi\in\mathcal{A}(y)$, notice that $\xi_{t}=\int_{0}^{t}\der\xi^{c}_{s}+\sum_{0\leq s\leq t}\Delta\xi_{s}$, where $\xi^{c}$ represents the continuous part of $\xi$. Then,   considering $\rho>0$ and $c>0$ as  the discount rate   and the transaction cost, respectively, we get that the  profit for extracting and selling at time $t$, with respect $\der\xi^{c}_{t}$, is given by $\expo^{-\rho t}[X_{t}-c]\der \xi^{c}_{t}$. However, if a large commodity selling occurs at $t$, i.e.,  $\Delta\xi_{t}\neq0$, it gives that  the profit at time $t$ is
\begin{align*}
	&\lim_{N\uparrow\infty}\sum\limits_{j=0}^{N-1}\expo^{-\rho t}(X_{t-}-\Delta S^{\tmn}_{t}+\Delta S^{\tmp}_{t}-c-j\alpha\varepsilon)\varepsilon\\
	&\qquad\qquad=\expo^{-\rho t}\int_{0}^{\Delta\xi_t}(X_{t-}-c-\Delta S^{\tmn}_{t}+\Delta S^{\tmn}_{t}-\alpha u) \der u\notag \\
	&\qquad\qquad=\expo^{-\rho t}\bigg \{[X_{t-}-c-\Delta S^{\tmn}_{t}+\Delta S^{\tmp}_{t}]\Delta\xi_t-\alpha\frac{[\Delta \xi_t]^{2}}{2}\bigg\}.
\end{align*}
{Based on the previously discussion, the total expected profit linked to the strategy $\xi\in\mathcal{A}(y)$ is expressed as}
\begin{align}\label{ENPV}
	\mathcal{J}(x,y,\xi)&\eqdef\E^{x,y}\bigg[\int_{0}^{\tau}\expo^{-\rho t}(X_{t}-c)\der \xi_{t}^{c}\notag\\
	&\quad+\sum\limits_{ 0\leq t\leq \tau}\expo^{-\rho t}\left[\left(X_{t{-}}-\Delta S^{\tmn}_{t}+\Delta S^{\tmp}_{t}-c \right)\Delta \xi_t-\alpha\frac{(\Delta\xi_t)^{2}}{2}\right]\bigg]\notag\\
	&=\E^{x,y}\bigg[\int_{0}^{\tau}\expo^{-\rho t}\left(X_{t}-c\right)\der \xi_{t}+
	\frac{\alpha}{2}\sum_{0\leq t\leq \tau}\expo^{-\rho t} [\Delta\xi_t]^{2}\bigg],
\end{align}
with $\tau=\inf\{t>0:Y_{t}<0\}$ {and $\E^{x,y}$ denotes the  expectation operator conditioned on the initial values $X_{0}=x$ and $Y_{0}=y$. Throughout this paper, the law of $(X,Y)$, when $X_{0}=x\in\R$ and $Y_{0}=y\geq0$, will be denoted by $\Pro^{x,y}$. For simplicity, we use $\Pro\equiv\Pro^{0,0}$ and $\E\equiv\E^{0,0}$ when $X_{0}=Y_{0}=0$.}

Given an initial inventory $y>0$, the  goal of the agent is to find a strategy $\xi^{\star}\in\mathcal{A}(y)$  (if one exists) such that 
\begin{equation}\label{fV}
	\mathcal{J}(x,y,\xi^{\star})=V(x,y):=\sup\limits_{\xi\in\mathcal{A}(y)}\mathcal{J}(x,y,\xi)\quad \text{for}\ x\geq0.
\end{equation}

\section{Preliminary results and HJB equation}
\label{sec3}

In this section, we will deduce the Hamilton-Jacobi-Bellman  equation (HJB) corresponding to the value function $V$ defined in  \eqref{fV}, and we will give a verification theorem. First, we will establish some properties of $V$.

\begin{prop}\label{PFV}
	The value function $V$ given by \eqref{fV} satisfies the following properties:
	\begin{enumerate}[label=(\arabic*),ref=\arabic*]
		\item  $V(x,0)=0$ for  $x\in\R$.
		\item $V(x,y)$ is increasing in  $y\geq 0$ for  $x\in\R$. 
		\item There exists a positive constant $K$ such that  
		\begin{equation}\label{p1V}
			0\leq V(x,y)\leq K y(1+y)(1+|x|) \quad \text{for}\ (x,y)\in \R\times \R^{+}_{*},
		\end{equation}
		where $\R^{+}_{*}\eqdef(0,\infty)$.
	\end{enumerate}
\end{prop}

\begin{proof}
	\begin{enumerate}[label=(\arabic*),ref=\arabic*]
		\item When $y=0$, it  is the scenario where the agent does not have to sell anything. In this case we have  $\xi\equiv0$ and $\mathcal{A}(0)=\{\xi\}$. Consequently, $V(x,0)=0$ for $x\in\R$.
		\item Let $0\leq y_1\leq y_2$.  By \eqref{cont.1}, it is easy to check that $\mathcal{A}(y_1)\subset\mathcal{A}(y_2)$. Hence,  $V(x,y_1)\leq V(x,y_2)$ for  $x\in\R$.
		\item Since for each $x\in\R$ fixed, the function $y\mapsto V(x,y)$ is increasing on $[0, \infty)$, it is  immediate that   $V(x,y)\ge 0$ for $(x,y)\in \R\times \R^{+}_{*}$. Now, consider  $x\in\R$, $y>0$ and  $\xi\in\mathcal{A}(y)$. Using \eqref{e2}--\eqref{e3}, it gives that 
		\begin{align}
			|\mathcal{J}(x,y,\xi)|
			&= \bigg|\E^{x,y}\bigg[\int_{0}^{\tau}\expo^{-\rho t}\left(X_{t}-c\right)\der \xi_{t}+
			\frac{\alpha}{2}\sum_{0\leq t\leq\tau}\expo^{-\rho t} [\Delta\xi_t]^{2}\bigg]\bigg| \notag\\
			&\leq[c+\alpha y]\E^{x,y}\bigg[\int_{0}^{\infty}\expo^{-\rho t}\der \xi _t\bigg]\notag\\
			&\quad+\frac{\alpha}{2}\E^{x,y}\Bigg[\sum_{t\geq0}\expo^{-\rho t}[\Delta\xi_t]^{2}\Bigg]+\E^{x,y}\bigg[\int_{0}^{\infty}\expo^{-\rho t}|X^{0}_{t}|\der \xi_t\bigg].\notag
		\end{align}
		Observe that $\int_{0}^{\infty}\der \xi_t\leq y$, due to  \eqref{e1}.  It implies that 
		\begin{equation}\label{pVd1}
			|\mathcal{J}(x,y,\xi)|\leq  y\bigg\{c+\E^{x}\bigg[\sup_{t\geq0}\{\expo^{-\rho t}|X^{0}_{t}|\}\bigg]\bigg\}+ \frac{3}{2}\alpha y^{2}.
		\end{equation}
		On the other hand, applying integration by parts on $\{\expo^{-\rho t}X_{t}^{0}:t\geq0\}$ (see \cite[Corollary 2, p. 68]{P2005}) and taking into account \eqref{eq1}, we have that 
		\begin{align*}
			|\expo^{-\rho t}X_{t}^{0}|&\leq |x|+|\mu|+\rho \int_{0}^{t}\expo^{-\rho s}|X_{s}^{0}|\der s\notag\\
			&\quad+\sigma\bigg|\int_{0}^{t}\expo^{-\rho s}\der W_s\bigg|+\int_{0}^{t}\expo^{-\rho s}\der S^{\tmn}_{s}+\int_{0}^{t}\expo^{-\rho s}\der S^{\tmp}_{s}.
		\end{align*}
		Then,
		\begin{align}\label{e6}
			&\E^{x}\left[\sup\limits_{t\geq0} \expo^{-\rho t}|X_{t}^{0}|\right]
			\leq |x|+|\mu|+\rho \int_{0}^{\infty}\expo^{-\rho t}\E^{x}\left[|X_{t}^{0}|\right]\der t\notag\\
			&+\E\bigg[\int_{0}^{\infty}\expo^{-\rho s}\der S^{\tmn}_{s}\bigg]+\E\bigg[\int_{0}^{\infty}\expo^{-\rho s}\der S^{\tmp}_{s}\bigg]+\sigma \E\bigg[\sup\limits_{t\geq0}\bigg|\int_{0}^{t}\expo^{-\rho s}\der W_s\bigg|\bigg].
		\end{align}
		Considering $\widetilde{N}^{u}$, with $u\in\{\tmn,\tmp\}$, as a random Poisson measure on  the space  
		\begin{equation}\label{e6.0}
			([0, \infty) \times [0,\infty), \mathcal{B}([0, \infty)) \times \mathcal{B}([0,\infty)), \der t \times \Pi_{u}(\der z)),
		\end{equation}
		where $\mathcal{B}([0,\infty))$ is the Borel sigma algebra on $[0,\infty)$ and $\Pi_{u}(\der z)\eqdef\lambda_{u}\der P_{u}(z)$, we get that  
		$\int_{0}^{t}\expo^{-\rho s}\der S^{u}_{s}=\int_{[0,t]}\int_{\{z\geq0\}}\expo^{-\rho s}z\widetilde{N}^{u}(\der s\times\der z).$
		By Theorem 2.7 in \cite{K2014}, we have that 
		\begin{align}\label{e4}
			\E\bigg[\int_{0}^{t}\expo^{-\rho s}\der S^{u}_{s}\bigg]&=\E\bigg[\int_{[0,t]}\int_{\{z\geq0\}}\expo^{-\rho s}z\widetilde{N}^{u}(\der s\times\der z)\bigg]\notag\\
			&=\frac{\lambda_{u}}{\rho}[1-\expo^{-\rho t}]\E[Z^{u}_{1}].
		\end{align}	
		Meanwhile, by Burholder-Davis-Gundy's inequality; see \cite[Theorem 3.28, p. 166]{KS1991}, it is well known that there exists a positive constant $C_{1}$ independent of time such that
		\begin{equation}\label{e5}
			\E\bigg[\sup\limits_{0\leq p\leq t}\bigg|\int_{0}^{p}\expo^{-\rho s}\der W_s\bigg|\bigg]\leq \frac{C_{1}}{[2\rho]^{1/2}}[1-\expo^{-2\rho t}]^{1/2}.
		\end{equation}
		Then, letting $t\rightarrow\infty$ in \eqref{e4}--\eqref{e5}, and by monotone convergence theorem, it follows that 
		\begin{equation}\label{e7}
			\E\bigg[\int_{0}^{\infty}\expo^{-\rho s}\der S^{u}_{s}\bigg]=\frac{\lambda_{u}}{\rho}\E[Z^{u}_{1}]\ \text{and}\  	\E\bigg[\sup\limits_{0\leq t}\bigg|\int_{0}^{t}\expo^{-\rho s}\der W_s\bigg|\bigg]\leq \frac{C_{1}}{[2\rho]^{1/2}}.
		\end{equation}	
		Using \eqref{e7} in \eqref{e6}, we get that  there  exists a constant $C_2>0$ such that
		\begin{equation}\label{pVd4}
			\E^{x}\bigg[\sup\limits_{t\geq0} \expo^{-\rho t}|X_{t}^{0}|\bigg]\leq C_2(1+|x|),
		\end{equation}
		due to $\int_{0}^{\infty}\expo^{-\rho t}\E^{x}\left[|X_{t}^{0}|\right]\der t\leq C_{3}[1+|x|]$ for some positive constant $C_{3}$. Finally, applying    \eqref{pVd4}  in \eqref{pVd1},    it is easy to check that \eqref{p1V} holds.\qed
	\end{enumerate}
\end{proof}

\subsection{Derivation of the HJB equation}
We will now derive the HJB equation related to the value function $V$. To begin, we introduce Lemma \ref{Exp1} to aid in this derivation and to establish the main result of this section, which is the verification theorem; see Proposition \ref{TV}. The proof of Lemma \ref{Exp1} is available in Appendix  \ref{A}.
\begin{lema}\label{Exp1}
	Let $f$ be a real valued function such that 	$\hol^{2,1}(\R\times\R^{*}_{+})$, which satisfies 
	\begin{align}\label{l1}
		0\leq f(x,y)\leq K y(1+y)(1+|x|) \quad \text{for}\ (x,y)\in \R\times \R^{+}_{*},
	\end{align}
	for some $K>0$. Then, for each $\xi\in\mathcal{A}(y)$,
	\begin{align}
		\label{ETV10.1}
		&\E^{x,y}[\expo^{-\rho t}f\left(X_{t},Y_{t}\right)-f(x,y) ]\notag\\
		&\qquad=\E^{x,y}\bigg[\int_{0}^{t}\expo^{-\rho s}\mathcal{D}f(X_s,Y_s)\der s-\sum_{0\leq s\leq t}\expo^{-\rho s}   \mathcal{I}[s;X,Y,f]\notag\\
		&\quad\qquad-\int_{0}^{t}\expo^{-\rho s}\left[\alpha \frac{\partial}{\partial x}f(X_{s},Y_s)+\frac{\partial}{\partial y}f(X_{s},Y_s)\right]\der \xi_s^{c}\bigg], 
	\end{align}
	where 
	\begin{align}
		\mathcal{I}[t;X,Y,f]&\eqdef 
		f(X_{s-}-\Delta S^{\tmn}_{s}+\Delta S^{\tmp}_{s}-\alpha \Delta \xi_s,Y_{s-}-\Delta\xi_s)\notag\\
		&\quad-f(X_{s-}-\Delta S^{\tmn}_{s}+\Delta S^{\tmp}_{s},Y_{s-})\notag\\
		&=\int_{0}^{\Delta \xi_t}(\alpha \frac{\partial}{\partial x}f(X_{t-}-\Delta S^{\tmn}_{t}+\Delta S^{\tmp}_{t}-\alpha u,Y_{t-}-u) \notag \\
		&\quad+f_{y}(X_{t-}-\Delta S^{\tmn}_{t}+\Delta S^{\tmp}_{t}-\alpha u,Y_{t-}-u))\der u,\label{g1}\\
		\mathcal{D}f(x,y)&\eqdef\frac{\sigma^{2}}{2}\frac{\partial^{2}}{\partial x^{2}}f(x,y)+\mu \frac{\partial}{\partial x}f(x,y)-(\rho+\lambda_{\tmn}+\lambda_{\tmp})f(x,y)\notag\\
		&\quad+\lambda_{\tmn}\int_{0}^{\infty}f(x-z,y)\der P_{\tmn}(z)+\lambda_{\tmp}\int_{0}^{\infty}f(x+z,y)\der P_{\tmp}(z).\label{generador}
	\end{align}
\end{lema}
Let us assume that the value function   $V$ is in $ \hol^{2,1}(\R\times \R_{*}^{+})$.  The agent faces two choices:  extract or wait. In the extraction scenario, the agent initially refrains from extracting for a short period $\Delta t$ and then proceeds according to the optimal extraction strategy (if one exists). This interim action, being suboptimal, is captured by an inequality as follows
\begin{equation}\label{EDH1}
	V(x,y)\geq \E^{x}\left[\expo^{-\rho\Delta t}V\left(X_{\Delta t},y\right)\right]\ \text{ for all } (x,y)\in \R\times \R_{*}^{+}.
\end{equation}
Since there is no selling on the interval time $[0,\Delta t]$ we have that $\xi_{t}\equiv0\in\mathcal{A}(y)$, $Y_{t}=y$ and $X_{t}=X^{0}_{t}$ for $t\in[0,\Delta t]$. Then, by Lemma \ref{Exp1}, we obtain  
\begin{align}
	\E^{x}\left[\expo^{-\rho\Delta t}V\left(X_{\Delta t},y\right)\right]-V(x,y)&=\E^{x}\left[\int_{0}^{\Delta t }\expo^{-\rho s}\mathcal{D}V(X_{s},y)\der s\right].
	\label{EDH5} 
\end{align}
Substituting \eqref{EDH5}  in \eqref{EDH1}, dividing by $\Delta t$ and then  letting $\Delta t\to 0 $, we have that 
\begin{equation}
	\label{eqDV}
	\mathcal{D}V(x,y)\leq 0.  
\end{equation}

In the second possibility,  we assume  that the agent decides to  extract a quantity of commodity $\varepsilon > 0$, sell  it in the market and then continue with the optimal extraction (assuming it exists). This action is associated with the inequality,
\begin{equation}\label{EDH6}
	V(x,y)\geq V(x-\alpha \varepsilon,y-\varepsilon)+(x-c)\varepsilon-\frac{1}{2}\alpha\varepsilon^2.
\end{equation}
Subtracting $V(x,y)$ from both sides of \eqref{EDH6}, dividing by $\varepsilon$, and then letting $\varepsilon \to 0$, we obtain
\begin{equation}\label{EDH7}
	-\alpha \frac{\partial}{\partial x}V(x,y)-\frac{\partial}{\partial y}V(x,y)+x-c\leq 0.
\end{equation}
Due to the  Markovian nature  of the problem,  it is expected that one of the  two actions will be optimal, and either  \eqref{eqDV} or \eqref{EDH7} will need to hold with equality. So, it  suggests that $V$ can be associated with a suitable solution $v$ of the HJB equation
\begin{equation}\label{eqHJB}
	\max\bigg\{\mathcal{D}v(x,y)	,-\alpha \frac{\partial}{\partial x}V(x,y)-\frac{\partial}{\partial y}V(x,y)+x-c\bigg\}=0
\end{equation}
for $(x,y)\in\R\times\R^{+}_{*}$, under the initial condition  $v(x,0)=0$ and the growth condition \eqref{p1V}.
\subsection{Verification Theorem}
\label{subsectionVT}
By the heuristic discussion  before, we divide $\R\times\R^{+}_{*}$ into two distinct sets,  called the \textit{no-selling } region and the \textit{selling}  region. These sets are defined, respectively, as follows
\begin{align}
	&\mathbb{W}\eqdef\bigg\{(x,y)\in\R\times\R^{+}_{*}: \mathcal{D}V(x,y)=0\notag\\
	&\hspace{4cm} \text{and  } -\alpha \frac{\partial}{\partial x}V(x,y)-\frac{\partial}{\partial y}V(x,y)<  c-x \bigg\},\label{RE}\\
	&\mathbb{S}\eqdef\bigg\{(x,y)\in\R\times\R^{+}_{*}: \mathcal{D}V(x,y)\leq0 \notag\\ 
	&\hspace{4cm} \text{and  } -\alpha \frac{\partial}{\partial x}V(x,y)-\frac{\partial}{\partial y}V(x,y)=  c-x \bigg\}.\label{RV}
\end{align}
We shall now introduce a verification theorem that establishes a connection between a smooth solution $v$ to the HJB equation \eqref{eqHJB} and the value function $V$. This theorem helps us in determining an optimal extraction strategy.   We  will denote by $\overline{\mathbb{W}}$ the topological closure of the set $\mathbb{W}$.

\begin{prop}[Verification theorem]\label{TV}
	Suppose that there exists a function $v:\R\times\R^{+}_{*}\longrightarrow \R$ such that $v\in \hol^{2,1}\left(\R\times\R^{+}_{*}\right)$ is the solution of the HJB equation \eqref{eqHJB} with boundary condition $v(x,0)=0$, increasing in $y$ and satisfies \eqref{l1}, then $v\geq V$  on $\R\times\R^{+}_{*}$. Furthermore,  suppose that for any initial value $(x,y)\in \R\times\R^{+}_{*}$, there exists $\xi^{\star}\in\mathcal{A}(y)$ such that  
	\begin{equation}\label{eo1}
		\left(X_t,Y_t\right)\in \overline{\mathbb{W}}\quad \text{for all}\ t\geq 0,\, \Pro\text{-c.s.},
	\end{equation}
	and
	\begin{equation}\label{eo2}
		\xi_{t}^{\star}=\int_{0}^{t}\mathbb{1}_{\{(X_t,Y_t)\in \mathcal{S}\}}\der \xi_{s}^{\star}\quad \text{for all}\ t\geq 0,\, \Pro\text{-c.s.}	
	\end{equation}
	Then  $v=V$ and $\mathcal{J}(x,y,\xi^{\star})=V(x,y)$ on $\R\times\R^{+}_{*}$, with $\mathcal{J}$ as in \eqref{ENPV}. Therefore, $\xi^{\star}$ is an optimal strategy.
\end{prop}

\begin{proof}
	Let  $\xi\in\mathcal{A}(y)$. Taking $\tau_{N}\eqdef\inf\{s\geq 0: |X_{s}|>N\}\wedge\tau $, with $N$ a positive integer large enough and recalling that $\tau=\inf\{t\geq0:Y_{t}<0\}$, by Lemma \ref{Exp1}, we get that
	\begin{align}
		\label{e8.0}
		v(x,y) 
		&=\E^{x,y}[\expo^{-\rho [\tau_{N}\wedge t]}v\left(X_{\tau_{N}\wedge t},Y_{\tau_{N}\wedge t}\right)]\notag\\
		&\quad-\E^{x,y}\bigg[\int_{0}^{\tau_{N}\wedge t}\expo^{-\rho s}\mathcal{D}v(X_s,Y_s)\der s-\sum_{0\leq s\leq \tau_{N}\wedge t}\expo^{-\rho s}   \mathcal{I}[s;X,Y,v]\notag\\
		&\quad-\int_{0}^{\tau_{N}\wedge t}\expo^{-\rho s}\left(\alpha \frac{\partial}{\partial x}V(X_{s},Y_s)+\frac{\partial}{\partial y}V(X_{s},Y_s)\right)\der \xi_s^{c}\bigg].
	\end{align}
	It is known that  $\mathcal{D}v(x,y)\leq0$	and $-\alpha \frac{\partial}{\partial x}V(x,y)-\frac{\partial}{\partial y}V(x,y)+x-c\leq0$ for $(x,y)\in\R\times\R^{*}_{+}$, due to \eqref{eqHJB}. From here and using \eqref{g1}, it follows that 
	\begin{align}\label{e8}
		v(x,y)&\geq
		\E^{x,y}[\expo^{-\rho [\tau_{N}\wedge t]}v\left(X_{\tau_{N}\wedge t},Y_{\tau_{N}\wedge t}\right)] +\E^{x,y}\bigg[\int_{0}^{\tau_{N}\wedge t}\expo^{-\rho s}\left(X_{s}-c\right)\der \xi_s^{c}\notag\\
		&\quad+\sum_{0\leq s\leq \tau_{N}\wedge t}\expo^{-\rho s} \bigg\{[X_{s-}-\Delta S^{\tmn}_{s}+\Delta S^{\tmp}_{s}-c]\Delta \xi_{s}-\frac{\alpha}{2} [\Delta\xi_{s}]^{2}\bigg\}\bigg]\notag\\
		&=\E^{x,y}[\expo^{-\rho [\tau_{N}\wedge t]}v\left(X_{\tau_{N}\wedge t},Y_{\tau_{N}\wedge t}\right)]\notag\\
		&\quad+\E^{x,y}\bigg[\int_{0}^{\tau_{N}\wedge t}\expo^{-\rho s}\left(X_{s}-c\right)\der \xi_s+\frac{\alpha}{2}\sum_{0\leq s\leq \tau_{N}\wedge t}\expo^{-\rho s} [\Delta\xi_{s}]^{2}\bigg\}\bigg].
	\end{align}
	On the other hand,  recalling that $\int_{0}^{\infty}\der \xi_t\leq y$, it gives that
	\begin{multline}\label{e8.1}
		\bigg|\int_{0}^{\tau_{N}\wedge t}\expo^{-\rho s}\left(X_{s}-c\right)\der \xi_s+\frac{\alpha}{2}\sum_{0\leq s\leq \tau_{N}\wedge t}\expo^{-\rho s} [\Delta\xi_{s}]^{2}\bigg|\\
		\leq y\bigg[c+\sup_{t\geq0}\{\expo^{-\rho t}|X^{0}_{t}|\}\bigg]+ \frac{3}{2}\alpha y^{2},
	\end{multline}	
	where the  term on the right side of \eqref{e8.1} is $\Pro^{x}$-integrable, because of \eqref{pVd4}.  Then, by the dominated convergence theorem and since $\tau_{N}\xrightarrow[N\rightarrow\infty]{}\tau$ $\Pro^{y}$-a.s., it gives that 
	\begin{align}\label{e8.2}
		\lim_{t,N\rightarrow\infty}\E^{x,y}\bigg[\int_{0}^{\tau_{N}\wedge t}\expo^{-\rho s}\left(X_{s}-c\right)\der \xi_s+\frac{\alpha}{2}\sum_{0\leq s\leq \tau_{N}\wedge t}\expo^{-\rho s} [\Delta\xi_{s}]^{2}\bigg]=\mathcal{J}(x,y,\xi).
	\end{align}
	Meanwhile, since $0\leq Y\leq y$  and taking into account \eqref{e3} and \eqref{l1}, we have that for some $K>0$ independent of $t$ and $N$, it yields that
	\begin{align}\label{e9.0}
		0&\leq \expo^{-\rho [\tau_{N}\wedge t]}v\left(X_{\tau_{N}\wedge t},Y_{\tau_{N}\wedge t}\right)\notag\\
		&\leq K y[1+y]\{[1+\alpha y]\expo^{-\rho[ \tau_{N}\wedge t]}+\expo^{-\rho[ \tau_{N}\wedge t]}|X^{0}_{\tau_{N}\wedge t}|\}.
	\end{align}
	Notice that the process $\{[1+\alpha y]\expo^{-\rho[ \tau_{N}\wedge t]}+\expo^{-\rho[ \tau_{N}\wedge t]}|X^{0}_{\tau_{N}\wedge t}|:t\geq0, N\geq0\}$ is $\Pro^{x,y}$-integrable, since by H\"older's inequality, we have that 
	\begin{multline}\label{e9.1}
		\E^{x,y}[[1+\alpha y]\expo^{-\rho[ \tau_{N}\wedge t]}+\expo^{-\rho[ \tau_{N}\wedge t]}|X^{0}_{\tau_{N}\wedge t}|]\\
		\leq [1+\alpha y]+\E^{x}\bigg[\sup_{0\leq t}\Big\{\expo^{-\rho t}|X^{0}_{t}|^{2}\Big\}\bigg]^{1/2}<\infty.
	\end{multline}
	To verify the truth of \eqref{e9.1}, observe that through the application of integration by parts and Itô's formula on $\expo^{-\rho t}[X^{0}_{t}]^{2}$, and then proceeding in a similar manner as in \eqref{e6}--\eqref{pVd4}, it can be verified that $\E^{x}[\sup_{0\leq t}\{\expo^{-\rho t}|X^{0}_{t}|^{2}\}] \leq C_{1}[1+x^{2}]$, for a certain positive constant $C_{1}$ independent of $t$ and $N$. 
	
	On the event $\{\tau<\infty\}$, by the seen before, invoking the dominated convergence theorem and considering 
	\begin{align*}
	\lim_{t,n\rightarrow\infty}\expo^{-\rho t}v((X_{\tau_{N}\wedge t},Y_{\tau_{N}\wedge t}))\uno_{\{\tau<\infty\}}
	=\expo^{-\rho\tau}v(X_{\tau},0)\uno_{\{\tau<\infty\}}=0,
	\end{align*}
	we get that  
	\begin{equation}\label{e9}
		\lim_{t,N\rightarrow\infty}\E^{x,y}[\expo^{-\rho [\tau_{N}\wedge t]}v\left(X_{\tau_{N}\wedge t},Y_{\tau_{N}\wedge t}\right)\uno_{\{\tau<\infty\}}]=0.
	\end{equation}
	On the event $\{\tau=\infty\}$, invoking the dominated convergence theorem again, it follows that 
	\begin{align}\label{e9.2}
		\lim_{n\rightarrow\infty}\E^{x,y}[\expo^{-\rho [\tau_{N}\wedge t]}v\left(X_{\tau_{N}\wedge t},Y_{\tau_{N}\wedge t}\right)\uno_{\{\tau=\infty\}}]=\E^{x,y}[\expo^{-\rho  t}v\left(X_{t},Y_{ t}\right)\uno_{\{\tau=\infty\}}].
	\end{align}
	Proceeding in a similar manner that in \eqref{e9.0}, we obtain that 
	\begin{align}\label{e9.3}
		0&\leq \E^{x,y}[\expo^{-\rho  t}v\left(X_{t},Y_{ t}\right)\uno_{\{\tau=\infty\}}]\notag\\
		&\leq K y[1+y]\expo^{-\rho t}\{1+\alpha y+ |x|+\mu t +\E[|W_{t}|]\}\notag\\
		&=K y[1+y]\expo^{-\rho t}\bigg\{ 1+\alpha y + |x|+\mu t +\bigg[\frac{2t}{\pi}\bigg]^{1/2}\bigg\}.
	\end{align}
	Letting $t\rightarrow\infty$ in \eqref{e9.2} and considering \eqref{e9.3}, it implies that 
	\begin{align}\label{e9.4}
		\lim_{t,n\rightarrow\infty}\E^{x,y}[\expo^{-\rho [\tau_{N}\wedge t]}v\left(X_{\tau_{N}\wedge t},Y_{\tau_{N}\wedge t}\right)\uno_{\{\tau=\infty\}}]=0.
	\end{align}
	Thus, letting $t,n\rightarrow\infty$ in \eqref{e8} a considering \eqref{e8.2}, \eqref{e9} and  \eqref{e9.4}, we get that $v(x,y)\geq \mathcal{J}(x,y,\xi)$, for all $xi\in\mathcal{A}(y)$.
	From here, we conclude that $v(x,y)\geq V(x,y)$ for  $(x,y)\in\R\times\R_{*}^{+}$. Now, let $\xi^{\star}$ be in $\mathcal{A}(y)$ satisfying \eqref{eo1} and \eqref{eo2}. Observe that \eqref{e8.0} holds for this case. Additionally, since  $\left(X_t,Y_t\right)\in \overline{\mathbb{W}} \text{ for all } t\geq 0$, $\Pro\text{-c.s.}$, and $\xi^{\star}$ acts when $(X_{t},Y_{t})\in  \mathbb{S}$, by definition of  $\mathbb{W}$ and $\mathbb{S}$ (see \eqref{RE} and \eqref{RV}, respectively), it gives that
	\begin{align}\label{e11}
		v(x,y)&=\E^{x,y}[\expo^{-\rho [\tau_{N}\wedge t]}v\left(X_{\tau_{N}\wedge t},Y_{\tau_{N}\wedge t}\right)]\notag\\
		&\quad+\E^{x,y}\bigg[\int_{0}^{\tau_{N}\wedge t}\expo^{-\rho s}\left(X_{s}-c\right)\der\xi^{\star}_s+\frac{\alpha}{2}\sum_{0\leq s\leq \tau_{N}\wedge t}\expo^{-\rho s} [\Delta\xi^{\star}_{s}]^{2}\bigg].
	\end{align}
	Letting $t\rightarrow\infty$ and $N\rightarrow\infty$ in \eqref{e11}, and arguing in a similar way as in \eqref{e8.1}--\eqref{e9.4}, it follows that $\mathcal{J}(x,y,\xi^{\star})=v(x,y)\geq V(x,y)$ for  $(x,y)\in\R\times\R_{*}^{+}$. Therefore, $\mathcal{J}(x,y,\xi^{\star})= V(x,y)$ for  $(x,y)\in\R\times\R_{*}^{+}$.\qed
\end{proof}

\section{Optimal extraction strategy  and construction of  a  solution to \eqref{eqHJB} }\label{S4}
In order to find an optimal extraction strategy for the problem outlined in \eqref{fV}, our initial focus will   be on the barrier strategies, which are defined as follows.  For a fixed $b>0$, the control process $\xi^{b}\eqdef\{\xi^{b}_{t}:t\geq0\}$ is considered as 
\begin{equation}\label{c1}
	\xi^{b}_{t}\eqdef Y^{b}_{t-}\wedge \frac{1}{\alpha}\sup_{0\leq s\leq t}[X^{0}_{s}-b]^{+}.
\end{equation}
It is worth noting that $\xi^{b}\in\mathcal{A}(y)$. Then, the controlled process $(X^{b},Y^{b})\eqdef\{(X^{b}_{t},Y^{b}_{t}):t\geq0\}$ associated with the strategy $\xi^{b}$ is given by $X^{b}_{t}=X^{0}_{t}-\alpha\int_{0}^{t}\der \xi^{b}_{t}$ and$Y^{b}_{t}=y-\xi^{b}_{t}$ for $t\geq0$. Hence, the objective of this section is to establish a constant $b^{\star}>0$ such that the expected net profit $v_{b^{\star}}(x,y)\eqdef\mathcal{J}(x,y,\xi^{b^{\star}})$ aligns with the value function $V$ given in \eqref{fV} by solving \eqref{eqHJB}. To accomplish this objective, assuming the existence of $b^{\star}>0$, to be determined later, and taking into account \eqref{c1}, let us consider  the following scenarios.
\begin{enumerate}
	\item[(i)] When $x\geq \alpha y+b^{\star}$, indicating that the  commodity  price is  high enough  to sell  available units,  $v_{b^{\star}}$ must follow the  expression for $(x,y)\in\mathcal{S}_{2}\eqdef\{(x,y):x\geq \alpha y+b^{\star}\ \text{and}\ y\geq0\}$,
	\begin{equation}\label{cond1}
		v_{b^{\star}}(x,y)=[x-c]y-\alpha\frac{y^{2}}{2}.
	\end{equation}
	
	\item[(ii)] When $b^{\star}\leq x<\alpha y+b^{\star}$, the decision is  to sell  $[x-b^{\star}]/\alpha$ units of commodities at time $0$. In this case,    $v_{b^{\star}}$ should satisfy
	\begin{align}\label{cond2}
		v_{b^{\star}}(x,y)&=[x-c]\frac{[x-b^{\star}]}{\alpha}-\frac{[x-b^{\star}]^{2}}{2\alpha}+v_{b^{\star}}(b^{\star},y-\mathbb{Y}_{b^{\star}}(x))\notag\\
		&=\frac{[x-c]^{2}-[c-b^{\star}]^{2}}{2\alpha}+v_{b^{\star}}(b^{\star},y-\mathbb{Y}_{b^{\star}}(x)) 
	\end{align}
	for $(x,y)\in\mathcal{S}_{1}\eqdef\{(x,y):b^{\star}\leq x<\alpha y+b^{\star}\ \text{and}\ y  > 0\}$. Here $\mathbb{Y}_{b^{\star}}(x)\eqdef(x-b^{\star})/\alpha$.
		\item[(iii)] When $x<b^{\star}$, indicating that the commodity price is too low to sell any units,  the decision is to wait. Over this condition we get that  $v_{b^{\star}}$ should satisfy 
		\begin{equation}\label{cond2.1}
			\mathcal{D}v_{b^{\star}}(x,y)=0\quad \text{ for}\ (x,y)\in\mathcal{W}\eqdef\{(x,y): x<b^{\star},\ y\geq0\}. 
		\end{equation}
	\end{enumerate}

	By taking first derivates w.r.t. $x$ and $y$ in \eqref{cond1}--\eqref{cond2},  observe that  
	\begin{equation}\label{v2}
		\alpha\parcial{}{x}v_{b^{\star}}(x,y)+\parcial{}{y}v_{b^{\star}}(x,y)=x-c\quad\text{for}\ (x,y)\in \mathcal{S}\eqdef\mathcal{S}_{1}\cup\mathcal{S}_{2}.
	\end{equation}	
	
	We will now derive an explicit solution of $v_{b^{\star}}$. To ensure a concise solution, it is essential for $v_{b^{\star}}$ to maintain smoothness at the boundary of each respective region mentioned above. Before presenting Proposition \ref{pro1},  let us introduce some essential parameters  (which depend on the parameters given in Section  \ref{SectionFP}). Taking
	\begin{equation}\label{eqaux}
		p(r)\eqdef\frac{[\sigma r]^{2}}{2}+\mu r-(\rho+\lambda_{\tmn}+\lambda_{\tmp})+\lambda_{\tmp}\sum_{i=1}^{m_{\tmp}}\frac{\omega_{i}^{\tmp}\beta_{i}^{\tmp}}{\beta_{i}^{\tmp}-r}+\lambda_{\tmn}\sum_{i=1}^{m_{\tmn}}\frac{\omega_{i}^{\tmn}\beta_{i}^{\tmn}}{r+\beta_{i}^{\tmn}}=0,
	\end{equation}
	by Theorem 3.1 in \cite{CK2011}, it is known  that \eqref{eqaux}   have at most $m_{\tmp}+1$ positive roots and  $m_{\tmn}+1$ negative roots. Additionally, denoting the positive and negative roots of \eqref{eqaux} by $r_{i}$ (with $i\in\{0,\dots, m_{\tmp} \}$) and $r^{\tmn}_{j}$ (with $j\in\{0,\dots, m_{\tmn} \}$) respectively,     they satisfy  the following monotonicity
	\begin{multline}\label{des}
		r^{\tmn}_{m_{\tmn}}<-\beta^{\tmn}_{m_{\tmn}}<\cdots<-\beta^{\tmn}_{2}<r^{\tmn}_{1}<-\beta^{\tmn}_{1}<r^{\tmn}_{0}<0\\
		<r_0<\beta_1^{\tmp}<r_1<\beta_2^{\tmp}<r_2<\beta_3^{\tmp}<\cdots<\beta_{m_{\tmp}}^{\tmp}<r_{m_{\tmp}}.
	\end{multline}
	Let $A_{n}$, with $n\in\mathcal{N}_{\tmp}\eqdef\{1,\dots,m_{\tmp}\}$, be a $(n+1)\times (n+1)$-matrix  such that 
	\begin{equation}\label{ma1}
		A_{n}=\begin{pmatrix}
			I_{n+1}\\
			B_{n}
		\end{pmatrix}
	\end{equation}
	where $B_n=(b_{ij})$ is a matrix of size $n\times (n+1)$ where $b_{ij}=\frac{\beta^{\tmp}_i}{\beta^{\tmp}_i-r_{j-1}}$, with $(i,j)\in\{1,\dots,n\}\times\{1,\dots,n+1\}$, and $I_{n+1}$ is a row with all elements equal to 1 of size $n+1$. The operators $\det[\,\cdot\,]$ and $\Cof_{i,j}[\,\cdot\,]$  indicate respectively the determinant and  the $(i,j)$ cofactor of any $n\times n$ matrix. 
	
	\begin{prop}\label{pro1}
		The function $v_{b^{\star}}$ given by 
		\begin{align}\label{v1.1}
			&v_{b^{\star}}(x,y)\notag\\
			&=\begin{cases}
				\sum_{j=1}^{m_{\tmp}+1}\frac{K_{j-1}}{\alpha r_{j-1}}[1-\expo^{-\alpha r_{j-1}y}]\expo^{r_{j-1}[x-b^{\star}]} &\text{if}\  x<b^{\star}\ \text{and }\ y\geq 0,\\
				\vspace{-0.45cm}&\\
				\sum_{j=1}^{m_{\tmp}+1}\frac{K_{j-1}}{\alpha r_{j-1}}\big[1-\expo^{-\alpha r_{j-1}[y-[x-b^{\star}]/\alpha]}\big]&\\
				\hspace{1.9cm}+\frac{1}{2\alpha}\big[[x-c]^{2}-[b^{\star}-c]^{2}\big] &\text{if}\ b^{\star}\leq x<\alpha y+b^{\star}\ \text{and}\ y>0,\\
				\vspace{-0.45cm}&\\
				[x-c]y-\frac{\alpha}{2}y^{2}&\text{if}\ x\geq \alpha y+b^{\star}\  \text{and}\ y\geq 0,
			\end{cases}
		\end{align}
		is a solution to \eqref{cond1}--\eqref{cond2.1}, where  $b^{\star}$ and $K_{j-1}$ are  such that
		\begin{align}
			b^{\star}&=  {c+\sum_{i=1}^{m_{\tmp}+1}\frac{1}{r_{i-1}}-\sum_{i=2}^{m_{\tmp}+1}\frac{1}{\beta_{i-1}^{\tmp}}>c},\label{b}\\
			K_{j-1}&=\frac{\mathcal{R}^{m_{\tmp}}\Cof_{1,j}[A_{m_{\tmp}}]}{[r_{j-1}]^{2}\det [A_{m_{\tmp}}]}>0\quad\text{for}\ j\in\{1,\dots,m_{\tmp}+1\},\label{K2}
		\end{align}
		with 
		\begin{align}\label{Cof3}
			\mathcal{R}^{n}\eqdef \frac{\bar{\vartheta}^{n}}{\underbar{$\vartheta$}^{n}},\quad \underbar{$\vartheta$}^{n}\eqdef\prod_{k=2}^{n+1}\beta^{\tmp}_{k-1}\quad\text{and}\quad\bar{\vartheta}^{n}\eqdef\prod_{k=1}^{n+1}r_{k-1}.
		\end{align}
	\end{prop}
	\begin{rem}
		The sequence $\{K_{j-1}\}_{j=1}^{m_{\tmp}+1}$ exhibits the following properties
		\begin{align}
			&\sum_{j=1}^{m_{\tmp}+1}K_{j-1}=b^{\star}-c,\label{cond1K}\\
			&\sum_{j=1}^{m_{\tmp}+1}r_{j-1 }K_{j-1}=1,\label{cond2K}\\
			&\sum_{j=1}^{m_{\tmp}+1}\frac{\beta_{k-1}^{\tmp}}{\beta_{k-1}^{\tmp}-r_{j-1}}K_{j-1}=b^{\star}-c+\frac{1}{\beta_{k-1}^{\tmp}},\label{cond3K}
		\end{align}
		where $k\in\{2,\dots,m_{\tmp}+1\}$.  Further elucidation on these properties can be found in Appendix \ref{B}.
	\end{rem}
	
	\section{Verification of optimality}
	\label{S5}
	In this section, we will verify that $v_{b^{\star}}$ given by \eqref{v1.1} satisfies the HJB equation \eqref{eqHJB}. This will demonstrate that the value function $V$ as defined in \eqref{fV} aligns with $v_{b^{\star}}$, and the strategy $\xi^{b^{\star}}$ given in \eqref{c1} (when $b=b^{\star}$) represents an optimal extraction strategy.
	\begin{theor}\label{teo1}
		Let  $v_{b^{\star}}$ be given by \eqref{v1.1}, then
		\begin{enumerate}[label=(\arabic*),ref=\arabic*]
			\item The function $v_{b^{\star}}$ belongs to $\hol^{2,1}\left(\R\times\R^{+}\right)$,  {is convex and increasing in $x$}, is increasing in $y$, and satisfies \eqref{l1}. Furthermore, $v_{b^{\star}}(x,0)=0$ for all $x\in\R$.
			\item The function $v_{b^{\star}}$ is   a solution to the  HJB equation \eqref{eqHJB}. Therefore,   the value function $V$, given by  \eqref{fV}, coincides with $v_{b^{\star}}$ and    $\xi^{b^{\star}}$, defined in \eqref{c1} when $b=b^{\star}$, is an optimal extraction strategy.
		\end{enumerate}
	\end{theor}
	
	\begin{proof}[Proof of Theorem \ref{teo1}. First part]
		We divide the proof in three steps. 
		
		\textit{Step 1.} First we shall prove that $v_{b^{\star}} \in \hol^{2,1}\left(\R\times\R^{+}\right)$. To prove that $v_{b^{\star}}$ is of class $\hol^2$, it is sufficient  to show this property only on  $\mathcal{A}\eqdef\{(b^{\star},y): y\geq 0\}\cup\{(\alpha y+b^{\star},y): y\geq 0\}$, because of $v_{b^{\star}}\in\hol^{2,1}((\R\times\R^{*})\setminus\mathcal{A})$. 
		
		\textit{Smooth fit at the variable $x$:}  Taking first derivative w.r.t. $x$ in \eqref{v1.1}, we have that 
		\begin{equation}\label{vx1W}
			\parcial{}{x}v_{b^{\star}}(x,y)=
			\begin{cases}
				\sum_{j=1}^{m_{\tmp}+1}\frac{K_{j-1}}{\alpha}[1-\expo^{-\alpha r_{j-1}\, y}]\expo^{r_{j-1}[x-b^{\star}]} &  \text{if } (x,y)\in\mathcal{W},  \\
				\vspace{-0.4cm}&\\
				-\sum_{j=1}^{m_{\tmp}+1}\frac{K_{j-1}}{\alpha}\expo^{r_{j-1}[x-b^{\star}-\alpha y]}+\frac{x-c}{\alpha} &  \text{if } (x,y)\in {\mathcal{S}_1}, \\
				y & \text{if } (x,y)\in {\mathcal{S}_2}. 
			\end{cases}
		\end{equation}
		From here and \eqref{cond1K}, we have that 
		\begin{align*}
			\parcial{}{x}v_{b^{\star}}(b^{\star}{+},y)&=-\sum_{j=1}^{m_{\tmp}+1}\frac{K_{j-1}}{\alpha}\expo^{-\alpha r_{j-1}\, y}+\frac{b^{\star}-c}{\alpha}\\
			&=\sum_{j=1}^{m_{\tmp}+1}\frac{K_{j-1}}{\alpha}[1-\expo^{-\alpha r_{j-1}\, y}]=\parcial{}{x}v_{b^{\star}}(b^{\star}{-},y).
		\end{align*}
		On the other hand,  
		\begin{align*}
			\parcial{}{x}v_{b^{\star}}([b^{\star}+\alpha y]{-},y)&=-\frac{1}{\alpha}\sum_{j=1}^{m_{\tmp}+1}K_i+\frac{1}{\alpha}(\alpha y+b^{\star}-c)\notag\\
			&=y=\parcial{}{x}v_{b^{\star}}([b^{\star}+\alpha y]{+},y).
		\end{align*}
		Hence, we have the continuity of  $\parcial{}{x}v_{b^{\star}}$ at $x=b^{\star}$. Now, calculating the second derivative w.r.t. $x$ and considering that \eqref{cond2K} is true, we see that
		\begin{align}\label{vxx1W}
			&\parcial{^{2}}{x^{2}}v_{b^{\star}}(x,y)\notag\\
			&\quad=\begin{cases}
				\sum_{j=1}^{m_{\tmp}+1}\frac{K_{j-1}\,r_{j-1}}{\alpha}[1-\expo^{-\alpha r_{j-1}\, y}]\expo^{r_{j-1}[x-b^{\star}]}, & \text{if } (x,y)\in\mathcal{W}, \\
				\vspace{-0.4cm}&\\
				\sum_{j=1}^{m_{\tmp}+1}\frac{K_{j-1}r_{j-1}}{\alpha}\Big[1-\expo^{r_{j-1}[x-b^{\star}-\alpha y]}\Big], &  \text{if } (x,y)\in {\mathcal{S}_1,}\\
				0& \text{if } (x,y)\in {\mathcal{S}_2}. 
			\end{cases}
		\end{align}
		Using \eqref{vxx1W}, it is straightforward to verify that $\parcial{^{2}}{x^{2}}v_{b^{\star}}(b^{\star}{+},y)=\parcial{^{2}}{x^{2}}v_{b^{\star}}(b^{\star}{-},y)$, and  $\parcial{^{2}}{x^{2}}v_{b^{\star}}([b^{\star}+\alpha y]{-},y)=0$. From this, it follows that $\parcial{^{2}}{x^{2}}v_{b^{\star}}$ is continuous at $x=b^{\star}$.  {Furthermore,  we have   $\parcial{}{x}v_{b^{\star}}(x,y)\geq0$ and $\parcial{^{2}}{x^{2}}v_{b^{\star}}(x,y)\geq0$ for $(x,y)\in\R\times\R^{+}$, which implies   that $v_{b^{\star}}$ is convex and increasing in $x$.}

		\textit{Smooth fit on $y$:} Now, we study the continuity of $\parcial{}{y}v_{b^{\star}}$. Calculating the first derivative w.r.t. $y$ in \eqref{v1.1}, 
		\begin{equation}\label{vy1W}
			\parcial{}{y}v_{b^{\star}}(x,y)=
			\begin{cases}
				\sum_{j=1}^{m_{\tmp}+1}K_{j-1}\expo^{r_{j-1}[x-b^{\star}-\alpha y]} & \text{if } (x,y)\in\mathcal{W}\cup{\mathcal{S}_1}, \\
				x-\alpha y-c& \text{if } (x,y)\in{\mathcal{S}_2},
			\end{cases}
		\end{equation}
		{ it follows easily that}  $\parcial{}{y}v_{b^{\star}}$ is continuous at $x=b^{\star}$. Meanwhile
		\begin{align*}
			\parcial{}{y}v_{b^{\star}}([b^{\star}+\alpha y]{-},y)&=\sum_{j=1}^{m_{\tmp}+1}K_{j-1} = b^{\star}-c =\parcial{}{y}v_{b^{\star}}([b^{\star}+\alpha y]{+},y).
		\end{align*}
		Therefore $\parcial{}{y} v_{b^{\star}} $ is continuous at $ x=b^{\star}+\alpha y$.
		
		\textit{Step 2.} As $K_i>0$ for all $i=0,\ldots,m_{\tmp}$, then $\parcial{}{y}v_{b^{\star}}
		(x,y)>0$ for all $(x,y)\in\mathcal{W}\cup {\mathcal{S}_1}$,due to \eqref{vy1W}. In the case that $(x,y)\in {\mathcal{S}_{2}}$, we get easily that    $\parcial{}{y}v_{b^{\star}}(x,y)=x-\alpha y-c\geq b^{\star}-c>0$. Thus, $v_{b^{\star}}$ is increasing on $y$.
		
		\textit{Step 3.} From \eqref{v1.1}, observe that if $(x,y)\in \mathcal{S}_2$ and  letting  $y\downarrow 0$, then $x\geq b^{\star}$ and $v_{b^{\star}}(x,0{+})=0$. Furthermore,    we have $v_{b^{\star}}(x,0)=0$ for  $x<b^{\star}$.
		Additionally, considering the monotonicity of $v_{b^{\star}}$ w.r.t. $y$, it readily follows that $v_{b^{\star}}(x,y)\geq0$ for $(x,y)\in\R\times\R^{+}_{*}$.
		
		We  shall  now establish  that $v_{b^{\star}}$ satisfies the   inequality on the right-hand side of \eqref{l1}. For  $(x,y)\in\mathcal{W}$, notice that  $\expo^{r_i(x-b^{\star})}<1$. Then, using that  
		\begin{equation}\label{v4}
			1+z\leq \expo^{z}\quad\text{for}\   z\in\R, 
		\end{equation}
		we obtain
		\begin{align*}
			v_{b^{\star}}(x,y)&\leq \sum_{j=1}^{m_{\tmp}+1}\frac{K_{j-1}}{\alpha r_{j-1}}(1-\expo^{-\alpha r_{j-1}\, y})\leq y\sum_{j=1}^{m_{\tmp}+1}K_{j-1}\leq \overline{K}_{1}\,y(1+y)(1+|x|),
		\end{align*}
		with $\overline{K}_{1}\eqdef b^{\star}-c$. If $(x,y)\in {\mathcal{S}_1 }$, by \eqref{v4}, we get that
		\begin{align*}
			v_{b^{\star}}(x,y)&\leq \sum_{j=1}^{m_\tmp+1}\frac{K_{j-1}}{\alpha}(\alpha y {-[x-b^{\star}]})+\frac{1}{2\alpha}\left[(x-c)^{2}-(b^{\star}-c)^{2}\right]\\
			&=   {y\overline{K}_{1}+\frac{1}{2\alpha}\left\{[x-c]^{2}-\big[2\overline{K}_{1}[x-c]+[\overline{K}_{1}]^{2}\big]\right\}}\\
			&= {y\overline{K}_{1}+\dfrac{1}{2\alpha}[x-b^{\star}]^{2}}\leq \overline{K}_{2}\,y(1+y)(1+|x|),
		\end{align*}
		with $\overline{K}_{2}\eqdef {\max\{\overline{K}_{1},\frac{\alpha}{2}\}}$. Finally, for $(x,y)\in {\mathcal{S}_2}$,  taking $\overline{K}_{3}\eqdef\max\{c,\frac{\alpha}{2}\}$,  it is easy to check that $v_{b^{\star}}(x,y)\leq \overline{K}_{3}y(1+y)(1+|x|)$.
		Therefore, by setting $K\eqdef\max\{\overline{K}_{1},\overline{K}_{2},\overline{K}_{3}\}$, we confirm  that $v_{b^{\star}}$ satisfies the   inequality on the right-hand side of \eqref{l1}.\qed
	\end{proof}
	
	To provide the proof of the second part of Theorem \ref{teo1}, we will introduce a crucial lemma that is essential for this purpose. The details of the proof for this lemma can be found in Appendix \ref{B}.
	\begin{lema}\label{ru}
		Let $r_i$ be the positive roots of \eqref{eqaux}, then
		\begin{align}
			&\frac{\sigma^{2}}{2}\mathcal{R}^{m_{\tmp}}+\mu-(\rho+\lambda_{\tmn})(b^{\star}-c)\notag\\
			&\quad+\lambda_{\tmp}\sum_{k=2}^{m_{\tmp}+1}\frac{\omega_{k-1}^{\tmp}}{\beta_{k-1}^{\tmp}}+\lambda_{\tmn}\sum_{k=2}^{m_{\tmn}+1}\sum_{j=1}^{m_{\tmp}+1}\frac{\omega_{k-1}^{\tmn}\,\beta_{k-1}^{\tmn}}{\beta_{k-1}^{\tmn}+r_{j-1}}K_{j-1}=0,\label{equ1}\\
			&\frac{\sigma^{2}}{2}-\rho\sum_{j=1}^{m_{\tmp}+1}\frac{K_{j-1}}{r_{j-1}}+\mu(b^{\star}-c)\notag\\
			&\quad-\lambda_{\tmn}\sum_{k=2}^{m_{\tmn}+1}\sum_{j=1}^{m_{\tmp}+1}\frac{\omega_{k-1}^{\tmn}}{\beta_{k-1}^{\tmn}+r_{j-1}}K_{j-1}=-\lambda_{\tmp}\sum_{k=2}^{m_{\tmp}+1}\sum_{j=1}^{m_{\tmp}+1}\frac{\omega_{k-1}^{\tmp}}{\beta_{k-1}^{\tmp}-r_{j-1}}K_{j-1},\label{equ2}\\
			&\Xi^{\tmn}_{k-1}\eqdef\sum_{j=1}^{m_{\tmp}+1}\frac{K_{j-1}\beta_{k-1}^{\tmn}}{\beta_{k-1}^{n}+r_{j-1}}-(b^{\star}-c)+\frac{1}{\beta_{k-1}^{\tmn}}=\sum_{j=1}^{m_{\tmp}+1}\frac{K_{j-1}[r_{j-1}]^{2}}{{\beta^{\tmn}_{k-1}}[\beta_{k-1}^{\tmn}+r_{j-1}]},\label{equ3}
		\end{align}
		where $k=2,\ldots,m_{\tmn}+1$.
	\end{lema}
	
	\begin{proof}[Proof of Theorem \ref{teo1}. Second part]
		By Proposition \ref{pro1}, it is knows that $v_{b^{\star}}$ satisfies \eqref{cond2.1} and \eqref{v2} on $\mathcal{W}$ and $\mathcal{S}$, respectively. Then, to prove that the function $v_{b^{\star}}$ satisfies \eqref{eqHJB}, it is sufficient to check  the following.
		\begin{enumerate}
			\item[(i)] If $(x,y)\in\mathcal{W}$, then
			\begin{equation}\label{v5}
				Tv_{b^{\star}}(x,y)\eqdef-\alpha \parcial{}{x}v_{b^{\star}}(x,y)-\parcial{}{y}v_{b^{\star}}(x,y)+x-c\leq0.
			\end{equation}
			\item[(ii)] If $(x,y)\in\mathcal{S}$, then 
			\begin{equation}\label{v6}
				\mathcal{D}v_{b^{\star}}(x,y)\leq0. 
			\end{equation}
			Recall that $\mathcal{D}$ is defined as in \eqref{generador}.
		\end{enumerate}
		Let $(x,y)$ be in $\mathcal{W}$. Taking into account \eqref{v1.1}, \eqref{cond1K}, \eqref{cond2K}, \eqref{vx1W}, and \eqref{vy1W}, it follows that 
		\begin{align*}
			Tv_{b^{\star}}(x,y)
			&\leq -\sum_{j=1}^{m_\tmp+1}K_{j-1}[r_{j-1}[x-b^{\star}]+1]+x-c {=0}.
		\end{align*}
		Thus \eqref{v5} holds on $\mathcal{W}$.
		
		The proof of \eqref{v6} is divided in two parts, when $(x,y)\in \mathcal{S}_1$ and $(x,y)\in \mathcal{S}_2$.
		
		\textit{Case 1.}  If $(x,y)\in\mathcal{S}_1$, using \eqref{eqaux}, \eqref{v1.1}, \eqref{vx1W} and \eqref{vxx1W}, notice that 
		\begin{align}
			&\frac{\sigma^{2}}{2}\parcial{^{2}}{x^{2}}v_{b^{\star}}(x,y)+\mu\parcial{}{x}v_{b^{\star}}(x,y)-(\rho+\lambda_\tmn+\lambda_\tmp)v_{b^{\star}}(x,y)\notag\\
			&\qquad=\sum_{j=1}^{m_{\tmp}+1}\frac{K_{j-1}}{\alpha r_{j-1}}\expo^{r_{j-1}[x-b^{\star}-\alpha y]}\notag\\
			&\qquad\quad\times\bigg[\lambda_{\tmp}\sum_{k=2}^{m_{\tmp}+1}\frac{\omega_{k-1}^{\tmp}\beta_{k-1}^{\tmp}}{\beta_{k-1}^{\tmp}-r_{j-1}}+\lambda_{\tmn}\sum_{k=2}^{m_{\tmn}+1}\frac{\omega_{k-1}^{\tmn}\beta_{k-1}^{\tmn}}{r_{j-1}+\beta_{k-1}^{\tmn}}\bigg]\notag\\
			&\qquad\quad+\frac{1}{\alpha}\bigg\{\frac{\sigma^{2}}{2}+\mu[x-c]\notag\\
			&\qquad\quad-[\rho+\lambda_\tmn+\lambda_\tmp]\bigg[\sum_{j=1}^{m_{\tmp}+1}\frac{K_{j-1}}{ r_{j-1}}+\dfrac{1}{2}[[x-c]^{2}-[b^{\star}-c]^{2}]\bigg]  \bigg\}.\label{eq3}
		\end{align}
		Meanwhile, considering again \eqref{v1.1} and using \eqref{cond3.0}, \eqref{cond3}, \eqref{cond3K}, and \eqref{equ3}, it can be checked that 
		\begin{align}
			&\int_{0}^{\infty}v_{b^{\star}}(x-z,y)\der P_{\tmn}(z)\notag\\ 
			&\quad=-\dfrac{1}{\alpha}\sum_{k=2}^{m_{\tmn}+1}\dfrac{\omega^{\tmn}_{k-1}\,\Xi^{\tmn}_{k-1}}{\beta^{\tmn}_{k-1}}\expo^{-\beta_{k-1}^{\tmn}[x-b^{\star}]}+ \sum_{j=1}^{m_{\tmp}+1}\frac{K_{j-1}}{\alpha r_{j-1}} \notag\\
			&\quad\quad- \sum_{j=1}^{m_{\tmp}+1}\frac{K_{j-1}}{\alpha r_{j-1}}\expo^{-r_{j-1} [b^{\star}-x+\alpha  y]}\sum_{k=2}^{m_{\tmn}+1}\frac{\beta_{k-1}^{\tmn}\omega^{\tmn}_{k-1}}{\beta^{\tmn}_{k-1} +r_{j-1}}\notag\\
			&\quad\quad+\sum_{k=2}^{m_{\tmn}+1}\frac{\omega^{\tmn}_{k-1}}{ \alpha  [\beta^{\tmn}_{k-1} ]^{2}}-[x-c]\sum_{k=2}^{m_{\tmn}+1}\dfrac{\omega^{\tmn}_{k-1}}{\alpha  \beta^{\tmn}_{k-1}}+\dfrac{[x-c]^{2}-[b^{\star}-c]^{2}}{2\alpha },\\
			&\int_{0}^{\infty}v_{b^{\star}}(x+z,y)\der P_{\tmp}(z)\notag\\
			&\quad=\sum_{j=1}^{m_{\tmp}+1}\frac{K_{j-1}}{\alpha r_{j-1}}-\sum_{j=1}^{m_{\tmp}+1}\frac{K_{j-1}}{\alpha r_{j-1}}\expo^{r_{j-1} [x-[b^{\star} +\alpha y]]}\sum_{k=2}^{m_{\tmp}+1}\frac{\omega^{\tmp}_{k-1}\beta^{\tmp}_{k-1}}{\beta^{\tmp} _{k-1}-r_{j-1}}\notag\\
			&\quad\quad+\sum_{k=2}^{m_{\tmp}+1}\frac{\omega^{\tmp}_{k-1}}{\alpha  [\beta^{\tmp}_{k-1}]^{2}} + [x-c]\sum_{k=2}^{m_{\tmp}+1}\frac{\omega^{\tmp}_{k-1} }{\alpha  \beta^{\tmp}_{k-1}}+\frac{[x-c]^{2}-[b^{\star}- c]^{2}}{2 \alpha}.\label{eq5}
		\end{align}
		Using \eqref{generador} and taking into account \eqref{eq3}--\eqref{eq5}, we see that  $\mathcal{D}v_{b^{\star}}(x,y)=H_{1}(x)$ for $(x,y)\in\mathcal{S}_{1}$, where
		\begin{align*}
			&H_{1}(x)\eqdef\frac{1}{\alpha}\bigg\{\frac{\sigma^{2}}{2}+\lambda_{\tmn}\sum_{k=2}^{m_{\tmn}+1}\frac{\omega^{\tmn}_{k-1}}{[\beta^{\tmn}_{k-1}]^{2}}+\lambda_{\tmp}\sum_{k=2}^{m_{\tmp}+1}\frac{\omega^{\tmp}_{k-1}}{ [\beta^{\tmp}_{k-1}]^{2}}\notag\\
			&\quad+[x-c]\bigg[\mu-\lambda_{\tmn}\sum_{k=2}^{m_{\tmn}+1}\dfrac{\omega^{\tmn}_{k-1}}{ \beta^{\tmn}_{k-1}}+\lambda_{\tmp}\sum_{k=2}^{m_{\tmp}+1}\frac{\omega^{\tmp}_{k-1} }{ \beta^{\tmp}_{k-1}}\bigg]\notag\\
			&\quad-\rho\bigg[\sum_{j=1}^{m_{\tmp}+1}\frac{K_{j-1}}{ r_{j-1}}+\dfrac{1}{2}[[x-c]^{2}-[b^{\star}-c]^{2}]\bigg]-\lambda_{\tmn}\sum_{k=2}^{m_{\tmn}+1}\dfrac{\omega^{\tmn}_{k-1}\,\Xi^{\tmn}_{k-1}}{ {\beta^{\tmn}_{k-1}}}\expo^{-\beta_{k-1}^{\tmn}[x-b^{\star}]}\bigg\}.
		\end{align*}
		On the other hand, observe that
		\begin{align*}
			&H_1(b^{\star}+)
			=\frac{1}{\alpha}\bigg\{ \frac{\sigma^{2}}{2}+\mu(b^{\star}-c)-\rho\sum_{j=1}^{m_\tmp+1}\frac{K_{j-1}}{r_{j-1}}\notag\\
			&\quad+\lambda_\tmp\sum_{k=2}^{m_\tmp+1}\frac{\omega^{\tmp}_{k-1}}{\beta_{k-1}^{\tmp}}\bigg[\dfrac{1}{\beta_{k-1}^{\tmp}}+b^{\star}-c\bigg]-\lambda_\tmn\sum_{k=2}^{m_{\tmn}+1}\frac{\omega^{\tmn}_{k-1}}{\beta_{k-1}^{\tmn}}\bigg[\Xi^{\tmn}_{k-1}-\frac{1}{\beta_{k-1}^{\tmn}}+b^{\star}-c\bigg]\bigg\}.
		\end{align*}
		Considering \eqref{cond1K}, \eqref{cond2K} and  \eqref{equ3}, notice that 
		\begin{equation}\label{eq6}
			\Xi^{\tmn}_{k-1}-\frac{1}{\beta_{k-1}^{\tmn}}+b^{\star}-c=\sum_{j=1}^{m_{\tmp}+1}\frac{K_{j-1}\beta^{\tmn}_{k-1}}{\beta^{\tmn}_{k-1}+r_{j+1}}
		\end{equation}
		for $k\in\{2,\dots,m_{\tmp}+1\}$. Then, from here and \eqref{equ2}, it follows that 
		\begin{align*}
			H_{1}(b^{\star}+)
			=&-\dfrac{\lambda_{\tmp}}{\alpha}\sum_{k=2}^{m_\tmp+1}\frac{\omega_{k-1}^{\tmp}}{\beta_{k-1}^{\tmp}}\left[\sum_{j=1}^{m_\tmp+1}\frac{\beta_{k-1}^{\tmp}K_{j-1}}{\beta_{k-1}^{\tmp}-r_{j-1}}-[b^{\star}-c]-\frac{1}{\beta_{k-1}^{\tmp}}\right]=0,
		\end{align*}
		where the last equality holds due to \eqref{cond3K}. Additionally, by \eqref{equ1} and using again \eqref{eq6}, we get that
		\begin{align*}
			H^{\prime}_1(x)&=\frac{1}{\alpha}\Bigg[ \mu-\lambda_\tmn\sum_{k=2}^{m_\tmn+1}\frac{\omega^{\tmn}_{k-1}}{\beta_{k-1}^{\tmn}}+\lambda_\tmp\sum_{k=2}^{m_\tmp+1}\frac{\omega^{\tmp}_{k-1}}{\beta_{k-1}^{\tmp}}\notag\\
			&\quad	-\rho(x-c)
			+\lambda_\tmn\sum_{k=2}^{m_\tmn+1}\omega_{k-1}^{\tmn}\Xi^{\tmn}_{k-1}\expo^{-\beta_{k-1}^{\tmn}(x-b^{\star})}\Bigg]\\
			&=-\frac{1}{\alpha}\Bigg[\lambda_\tmn\sum_{k=2}^{m_\tmn+1}\omega_{k-1}^{\tmn}\Xi^{\tmn}_{k-1}[1-\expo^{-\beta_{k-1}^{\tmn}(x-b^{\star})}]+\rho(x-b^{\star})+\frac{\sigma^{2}}{2}\mathcal{R}^{m_{\tmp}}\Bigg].
		\end{align*}	
		From here and since $\mathcal{R}^{m_{\tmp}}>0$, $\Xi^{\tmn}_{k-1}>0$ for each $k\in\{1,\ldots,m_\tmn+1\}$ (due to \eqref{equ3}), and $x>b^{\star}$, it yields that $H'_{1}<0$ on $ [b^{\star},\infty)$. Therefore,  $H_1(x)$ is decreasing for $x>b^{\star}$ satisfying $H_1(b^{\star}+)=0$. Hence $\mathcal{D}v_{b^{\star}}(x,y)=H_{1}(x)<0$ for   $(x,y)\in\mathcal{S}_1$.
		
		\textit{Case 2.}
		If $(x,y)\in\mathcal{S}_2$, considering \eqref{v1.1}, notice that 
		\begin{multline}\label{eq7}
			\frac{\sigma^{2}}{2}\parcial{^{2}}{x^{2}}v_{b^{\star}}(x,y)+\mu\parcial{}{x}v_{b^{\star}}(x,y)-(\rho+\lambda_{\tmn}+\lambda_{\tmp})v_{b^{\star}}(x,y)\\
			=\mu y-[\rho+\lambda_{\tmn}+\lambda_{\tmp}]\left[[x-c]y-\frac{\alpha}{2}y^{2}\right].
		\end{multline}
		Meanwhile, considering again \eqref{v1.1} and using \eqref{cond3.0}, \eqref{cond3}, and \eqref{cond3K}, it can be checked that 
		\begin{align}
			\int_{0}^{\infty}v_{b^{\star}}(x-z,y)\der P_{\tmn}(z)&=[x-c]y-\dfrac{\alpha  y^{2}}{2}-y\sum_{k=2}^{m_{\tmn}+1}\frac{\omega^{\tmn}_{k-1}}{ \beta^{\tmn}_{k-1} }\notag\\
			&\quad\hspace{-0.5cm}-\dfrac{1}{\alpha}\sum_{k=2}^{m_{\tmn}+1}\frac{\omega^{\tmn}_{k-1}\Xi^{\tmn}_{k-1}}{\beta^{\tmn}_{k-1}}\expo^{-\beta_{k-1}^{\tmn}[x-b^{\star}]}[1-\expo^{-\beta_{k-1}^{\tmn}\alpha  y}],\\
			\int_{0}^{\infty}v_{b^{\star}}(x+\omega,y)\der P_{\tmp}(z)
			&=[x-c]y-\frac{\alpha}{2}y^{2}+y\sum_{k=2}^{m_{\tmp}+1}\frac{\omega_{k-1}^{\tmp}}{\beta_{k-1}^{\tmp}}.\label{eq8}
		\end{align}
		From \eqref{eq7}--\eqref{eq8}, we get that  $\mathcal{D}v_{b^{\star}}(x,y)=H_{2}(x,y)$ for $(x,y)\in\mathcal{S}_{2}$, with
		\begin{align*}
			H_2(x,y)&\eqdef\mu y-\rho\bigg[[x-c]y-\alpha\frac{y^{2}}{2}\bigg]+y\bigg[\lambda_\tmp\sum_{k=2}^{m_\tmp+1}\frac{\omega_{k-1}^{\tmp}}{\beta_{k-1}^{\tmp}}-\lambda_\tmn\sum_{k=2}^{m_\tmn+1}\frac{\omega_{k-1}^{\tmn}}{\beta_{k-1}^{\tmn}}\bigg]\\
			&\quad-\frac{\lambda_\tmn}{\alpha}\sum_{k=2}^{m_\tmn+1}\frac{\omega_{k-1}^{\tmn}\Xi^{\tmn}_{k-1}}{\beta_{k-1}^{\tmn}}\expo^{-\beta_k^{\tmn}[x-b^{\star}]}[1-\expo^{-\alpha \beta_{k-1}^{\tmn}y}]\notag\\
			&=-\rho y\bigg[x-b^{\star}-\alpha\frac{y}{2}\bigg]-y\mathcal{R}^{m_{\tmp}}\frac{\sigma}{2}\notag\\ &\quad-\lambda_\tmn\sum_{k=2}^{m_{\tmn}+1}\Xi^{\tmn}_{k-1}\omega_{k-1}^{\tmn}\bigg[y+\frac{[1-\expo^{-\alpha\beta_k^{\tmn}y}]\expo^{-\beta_k^{\tmn}[x-b^{\star}]}}{\alpha\beta_{k-1}^{\tmn}}\bigg],
		\end{align*}
		where the last equality is obtained due to  \eqref{equ1} and \eqref{equ3}. Therefore, since $x>b^{\star}+\alpha y$, we conclude  $\mathcal{D}v_{b^{\star}}(x,y)=H_2(x,y)<0$ for all $(x,y)\in\mathcal{S}_2$.
		
		Since $\xi^{b^{\star}}$ (as in \eqref{c1} when $b=b^{\star}$) satisfies the conditions \eqref{eo1} and \eqref{eo2}, and $v_{b^{\star}}$ is the solution of the HJB equation \eqref{eqHJB} which satisfies the hypotheses of Proposition \ref{TV}, we conclude that $v_{b^{\star}}(x,y)=V(x,y)$  for all $(x,y)\in\R\times\R^{+}_{*}$.\qed
	\end{proof}

	\section{A related optimal stopping problem}\label{stop1}
	
	In this section, we aim to explore the relationship between the directional derivative $u:=\alpha \frac{\partial}{\partial x}V+\frac{\partial}{\partial y}V$, with $V$ as in \eqref{v1.1}, and the stopping time problem given by
	\begin{equation}\label{st1}
		U(x)\eqdef\sup _{\tau \in\mathscr{S}} \mathbb{E}^{x}\big[\expo^{-\tau \rho}\left[X_\tau-c\right]\big], \quad x \in \R,
	\end{equation}
	 where $\mathscr{S}$ denotes the class of all $\F$-stopping times. By applying dynamic programming, we derive the variational inequality associated with \eqref{st1}, which takes the following form
	 \begin{equation}\label{EDu1}
	 	\max\left\{\Gamma w(x),x-c-w(x)\right\}=0\quad \text{for}\ x\in\R \blue{,}
	 \end{equation}	 
	 where 
	\begin{align*}
		\Gamma w(x)&\eqdef\frac{\sigma^{2}}{2}w''(x)+\mu w'(x)-(\rho+\lambda_{\tmn}+\lambda_{\tmp})w(x)\notag\\
		&\quad+\lambda_{\tmn}\int_{0}^{\infty}w(x-z)\der P_{\tmn}(z)+\lambda_{\tmp}\int_{0}^{\infty}w(x+z)\der P_{\tmp}(z).
	\end{align*}	 
	 Here, $w'$ and $w''$ represent the first and second derivatives  of $w$, respectively. 
	 
	 Before presenting the main result of this section,  Proposition \ref{st2}, notice that from \eqref{vx1W} and \eqref{vy1W}, it follows that
	 \begin{equation}\label{st1.1}
	 	\alpha \frac{\partial}{\partial x}V(x,y)+\frac{\partial}{\partial y}V_{y}(x,y)=
	 	\begin{cases}
	 	\sum_{j=1}^{m_{\tmp}+1}K_{j-1}\expo^{r_{j-1}[x-b^{\star}]} & \text{ if   } x<b^{\star},\\
	 	x-c & \text{ if } x\geq b^{\star},
	 	\end{cases}
	 	 \end{equation} 
	 with $b^{\star}$ as in \eqref{b}. This  shows  that the directional derivative $u$ depends solely on the variable $x$.
	 \begin{prop}\label{st2}
	 	The directional derivative $u$ is a $[\sob^{2,
	 	\infty}_{\loc}(\R)\cap\hol^{1}(\R)]$-solution to the variational inequality \eqref{EDu1}. It is both increasing and convex. Consequently, $u$ coincides with the value function $U$  defined in \eqref{st1} across $\R$. Furthermore, the stopping time 
	 	\begin{equation}\label{tau1}
	 		\tau(x)\eqdef\inf\{t\geq0:X_{t}\geq b^{\star}\}
	 	\end{equation}
	 	is optimal for the problem stated in \eqref{st1}.
	 \end{prop}
	 \begin{rem}
	 	The stopping time $\tau(x )$ represents the optimal moment to extract an additional unit of the commodity. At $\tau(x)$  the underlying process satisfies the economic condition that the marginal expected optimal profit, $u$ equals the marginal instantaneous net profit from extraction, $x-c$.
	 \end{rem}
	 The Sobolev space $\mathrm{W}^{k, p}(\mathcal{O})$, with $\mathcal{O}\subseteq\R$ and $1\leq p<\infty$, is defined as the class of functions $f \in \Lp^p(\mathcal{O})$ that possess weak or distributional partial derivatives $f^{(m)}$ (see \cite[p. 22]{AF2003}) and have finite norm defined by $\|f\|_{\sob^{k, p}(\mathcal{O})}^p=\sum_{m=0}^k \left\| f^{(m)}\right\|_{\Lp^p(\mathcal{O})}^p=\sum_{m=0}^k \int_{\mathcal{O}}|f^{(m)}|^{p}\der x.$ Then, the space $\mathrm{W}_{\text {loc }}^{k, p}(\mathcal{O})$ consists of functions whose $\mathrm{W}^{k, p}$-norm is finite on any compact subset of $\mathcal{O}$. When $p=\infty$, the $\mathrm{W}^{k, \infty}$-norm is $||f||_{\sob^{k,\infty}(\set)}\eqdef\max_{m\in\{0,\dots,k\}}\{||f^{(m)}||_{\Lp^{\infty}(\set)}\}$. Moreover, the Sobolev and Lipschitz spaces are related. In particular, $\mathrm{W}_{\text {loc }}^{k, \infty}(\mathcal{O})=\mathrm{C}_{\text {loc }}^{k-1,1}(\mathcal{O})$ and $\mathrm{W}^{k, \infty}(\mathcal{O})=$ $\mathrm{C}^{k-1,1}(\overline{\mathcal{O}})$; for further details, see, e.g. \cite{AF2003}. 
	 
	 \begin{proof}[Proof of Proposition \ref{st2}]  
	 	We will begin by proving that the directional derivative $u$ as in \eqref{st1.1} is a $[\sob^{2,\infty}_{\loc}(\R)\cap\hol^{1}(\R)]$-solution to \eqref{EDu1}. Notice that  $u\in\hol^{1} (\R)$, due to \eqref{cond2K}. However,  $u\in\hol^{2} (\R\setminus\{b^{\star}\})$, because of $u''(b^{\star}-)=\sum_{j=0}^{m_{\tmp}+1}K_{j-1}r^{2}_{j-1}>0=u''(b^{\star}+)$.  Additionally, we have that $u,u',u''\in\Lp^{\infty}_{\loc}(\R)$, and that $u$ is an increasing and convex function on $\R$.
	 	
	 	Let us consider $x<b^{\star}$. We can express $\Gamma u$  as follows
	 		\begin{align}\label{DuW}
	 			\Gamma u(x)&=\frac{\sigma^{2}}{2}\sum_{j=1}^{m_{\tmp}+1}r_{j-1}^{2}K_{j-1}\expo^{r_{j-1}[x-b^{\star}]}\notag\\
	 			&\quad+\mu\sum_{j=1}^{m_{\tmp}+1}r_{j-1}K_{j-1}\expo^{r_{j-1}[x-b^{\star}]}-[\rho+\lambda_{\tmp}+\lambda_{\tmp}]\sum_{j=1}^{m_{\tmp}+1}K_{j-1}\expo^{r_{j-1}[x-b^{\star}]}\notag\\
	 			&\quad+\lambda_{\tmn}\int_{0}^{\infty}u(x-z)\der P_{\tmn}(z)+\lambda_{\tmp}\int_{0}^{\infty}u(x+z)\der P_{\tmp}(z).
	 		\end{align}
	 		On the other hand, using \eqref{cond3.0}, observe that
	 		\begin{align}
	 			\int_{0}^{\infty}u(x-z)\der P_{\tmn}(z)
	 			&=\sum_{j=1}^{m_\tmp+1}\sum_{i=2}^{m_\tmn+1}\frac{\omega_{i-1}^{\tmn}\beta_{i-1}^{\tmn}}{\beta_{i-1}^{\tmn}+r_{j-1}}K_{j-1}\expo^{r_{j-1}[x-b^{\star}]},\label{InuW}\\
	 			\int_{0}^{\infty}u(x+z)\der P_{\tmp}(z)&=\int_{0}^{b^{\star}-x}u(x+z)\der P_{\tmp}(z)+\int_{b^{\star}-x}^{\infty}u(x+z)\der P_{\tmp}(z)\notag\\
	 			&=\sum_{j=1}^{m_{\tmp}+1}\sum_{i=2}^{m_{\tmp}+1}\frac{\omega_{i-1}^{\tmp}\beta_{i-1}^{\tmp}}{\beta_{i-1}^{\tmp}-r_{j-1}}K_{j-1}\expo^{r_{j-1}[x-b^{\star}]}.\label{IpuW}
	 		\end{align}
	 		Applying \eqref{InuW}--\eqref{IpuW} in \eqref{DuW} and considering that $r_{j-1}$ satisfies  \eqref{eqaux} for each $j\in\{1,\dots,m_{\tmp}+1\}$, it gives that 
	 		\begin{align}\label{st5}
	 			\Gamma u(x)&=\sum_{j=1}^{m_{\tmp}+1}K_{j-1}\expo^{r_{j-1}[x-b^{\star}]}\Bigg[\frac{\sigma^{2}}{2}r_{j-1}^{2}+\mu r_{j-1}-[\rho+\lambda_{\tmp}+\lambda_{\tmp}]\notag\\
	 			&\quad+\lambda_n\sum_{i=2}^{m_{\tmn}+1}\frac{\omega_{i-1}^{\tmn}\beta_{i-1}^{\tmn}}{\beta_{i-1}^{\tmn}+r_{j-1}}+\lambda_{\tmp}\sum_{i=2}^{m_{\tmp}+1}\frac{\omega_{i-1}^{\tmp}\beta_{i-1}^{\tmp}}{\beta_{i-1}^{\tmp}-r_{j-1}}\Bigg]=0.
	 		\end{align}
	 		Meanwhile, as  discussed in the second part of the proof of Theorem \ref{teo1},  it is known that  $x-c-u(x)<0$.  
	 		
	 		Let us  now consider  $x\geq b^{\star}$. It has immediately that $x-c-u(x)=0$ due to  \eqref{v2}. We will prove that $\Gamma u(x)<0$.   Since 
	 		\begin{align*}
	 			\int_{0}^{\infty}u(x+z)\der P_{\tmp}(z)&=
	 			x-c+\sum_{i=2}^{m_{\tmp}+1}\frac{\omega_{i-1}^{\tmp}}{\beta_{i-1}^{\tmp}},\\
	 				\int_{0}^{\infty}u(x-z)\der P_{\tmn}(z)&=\int_{0}^{x-b^{\star}}u(x-z)\der P_{\tmn}(z)+\int_{x-b^{\star}}^{\infty}u(x-z)\der P_{\tmn}(z)\notag\\
	 				&=\sum_{i=2}^{m_\tmn+1}\omega_{i-1}^{\tmn}\expo^{-\beta_{i-1}^{\tmn}[x-b^{\star}]}\Bigg[\frac{1}{\beta_{i-1}^{\tmn}}-(b^{\star}-c)\notag\\
	 				&\quad+\sum_{j=1}^{m_{\tmp}+1}\frac{\beta_{i-1}^{\tmn}}{\beta_{i-1}^{\tmn}+r_{j-1}}K_{j-1}\Bigg]+x-c-\sum_{i=2}^{m_{\tmn}+1}\frac{\omega_{i-1}^{\tmn}}{\beta_{i-1}^{\tmn}},
	 		\end{align*}
	 		we get that 
	 		\begin{align}\label{Ux}
	 			\Gamma u(x)&=\mu-\rho[x-c]-\lambda_{\tmn}\sum_{i=2}^{m_{\tmn}+1}\frac{\omega_{i-1}^{\tmn}}{\beta_{i-1}^{\tmn}}+\lambda_{\tmp}\sum_{i=2}^{m_{\tmp}+1}\frac{\omega_{i-1}^{\tmp}}{\beta_{i-1}^{\tmp}}\notag\\
	 			&\quad+\lambda_{\tmn}\sum_{i=2}^{m_\tmn+1}\omega_{i-1}^{\tmn}\expo^{-\beta_{i-1}^{\tmn}(x-b^{\star})}\Bigg[\frac{1}{\beta_{i-1}^{\tmn}}-[b^{\star}-c]+\sum_{j=1}^{m_{\tmp}+1}\frac{\beta_{i-1}^{\tmn}}{\beta_{i-1}^{\tmn}+r_{j-1}}K_{j-1}\Bigg].	
	 		\end{align}
	 		Considering \eqref{equ1}, we observe that
	 		\begin{align*}
	 			\Gamma u(b^{\star}+)&=\mu-(\rho+\lambda_{\tmn})(b^{\star}-c)\notag\\
	 			&\quad+\lambda_{\tmp}\sum_{i=1}^{m_{\tmp}+1}\frac{\omega_{i-1}^{\tmp}}{\beta_{i-1}^{\tmp}}+\lambda_{\tmp}\sum_{i=2}^{m_\tmn+1}
	 			\sum_{j=1}^{m_{\tmp}+1}\frac{\omega_{i-1}^{\tmn}\beta_{i-1}^{\tmn}}{\beta_{i-1}^{\tmn}+r_{j-1}}K_{j-1}=-\frac{\sigma^{2}}{2}\mathcal{R}^{m_{\tmp}}<0.
	 		\end{align*}
	 		Additionally,  applying first derivative to \eqref{Ux} and taking into account \eqref{cond1K}--\eqref{cond2K} , it give that
	 		\begin{align*}
	 			[\Gamma u]^{\prime}(x)
	 			&=-\rho-\lambda_{\tmn}\sum_{i=2}^{m_{\tmn}+1}\omega_{i-1}^{\tmp}\expo^{-\beta_{i-1}^{\tmp}(x-b^{\star})}\Bigg[\sum_{j=1}^{m_{\tmp}+1}r_{j-1}K_{j-1}\notag\\
	 			&\quad-\beta_{i-1}^{\tmp}\sum_{j=1}^{m_{\tmp}+1}K_{j-1}+\sum_{j=1}^{m_{\tmp}+1}\frac{[\beta_{i-1}^{\tmn}]^{2}}{\beta_{i-1}^{\tmn}+r_{j-1}}K_{j-1}\Bigg]\\
	 			&=-\rho-\sum_{i=2}^{m_\tmn+1}\sum_{j=1}^{m_{\tmp}+1}\expo^{-\beta_{i-1}^{\tmn}(x-b^{\star})}\frac{\omega_{i-1}^{\tmn}r_{j-1}^{2}}{\beta_{i-1}^{\tmn}+r_{j-1}}K_{j-1}<0.
	 		\end{align*}
	 		This implies that $\Gamma u$ is a negative  and decreasing function on $[b^{\star},\infty)$. Since $\Gamma u(b^{\star}-)>\Gamma u(b^{\star}+)$, indicating  a discontinuity at $b^{\star}$, we conclude that $u$ as in \eqref{st1.1} is an increasing and convex function which is a $[W_{\loc}^{2,\infty}\left(\R\right)\cap\hol^{1}(\R)]$-solution to the variational inequality \eqref{EDu1}.
	 		
	 		To verify that $u$ coincides with $U$ on $\R$  and that $\tau^{\star}\eqdef\tau(x)$ as in \eqref{tau1} is an optimal stopping for the problem \eqref{st1}, it suffices  to verify the following two condition: (i) $u(x)=\E^{x}\big[\expo^{-\tau^{\star}\rho}[X_{\tau^{\star}}-c]\big]$, and   (ii) $u(x)\geq\E^{x}\big[\expo^{-\tau\rho}[X_{\tau}-c]\big]$  for any $\tau\in\mathscr{S}$.
	 		
	 		Applying integration by parts to $\expo^{-\rho t}u(X_{t})$ and Meyer-It\^o  formula to $u(X_{t})$  (due to the convexity of $u$, the continuity of $u''$ on $\R\setminus\{b^{\star}\}$ and the existence of  finite lateral limits of $u''$ at $b^{\star}$);  see Corollaries II.2, IV.1 and Theorem IV.70 \cite{P2005},  then	 
	 		\begin{align}\label{st3}
	 			u(x)&=\expo^{-\rho t}u(X_{t})-\int_{0}^{t}\expo^{-\rho s}\Gamma u(X_{s-})\der s-\widetilde{M}_{t},
	 		\end{align}
	 		where 
	 		\begin{align*}
	 			\widetilde{M}_{t}&\eqdef\int_{0}^{t}\sigma\expo^{-\rho s}u'(X_{s-})\der W_{s}\notag\\
	 			&\quad+ \int_{0}^{t}\int_{0}^{\infty}\expo^{-\rho s}\big(u(X_{s-}-z)-u(X_{s-})\big)(\widetilde{N}^{\tmn}(dz\times \der s)-\lambda_{\tmn} {\der P}_{\tmn}(z)\der s)\notag\\
	 			&\quad+ \int_{0}^{t}\int_{0}^{\infty}\expo^{-\rho s}\big(u(X_{s-}+z)-u(X_{s-})\big)(\widetilde{N}^{\tmp}(dz\times \der s)-\lambda_{\tmp} {\der P}_{\tmp}(z)\der s)
			\end{align*} 		
	 		is a zero-mean local martingale. Taking a localizing sequence $\{\tau_{n}\}_{n\geq0}\subset\mathscr{S}$, the reader may verify, using the explicit form of $u$ given in \eqref{st1.1} , that  $|\widetilde{M}_{t\wedge\tau_{n}\wedge\tau^{\star}}|\leq Z$ for all $n\geq0$ and $t\geq0$, for some integrable positive random variable $Z$.  Hence  $\{\widetilde{M}_{t\wedge\tau_{n}\wedge\tau^{\star}}:t\geq0\}$ is a zero $\Pro^{x}$-martingale. Taking expected value and replacing $t$ by  $t\wedge\tau_{n}\wedge\tau^{\star}$ in \eqref{st3}, we obtain 
	 		\begin{align*}
	 			u(x)&=\E^{x}[\expo^{-\rho [t\wedge\tau_{n}\wedge\tau^{\star}]}u(X_{t\wedge\tau_{n}\wedge\tau^{\star}})]-\E^{x}\Bigg[\int_{0}^{t\wedge\tau_{n}\wedge\tau^{\star}}\expo^{-\rho s}\Gamma u(X_{s-})\der s\Bigg].
	 		\end{align*}
	 		On the other hand, observe that $\lim_{t\rightarrow\infty}\lim_{n\rightarrow\infty}[t\wedge\tau_{n}\wedge\tau^{\star}]=\tau^{\star}$ $\Pro^{x}$-a.s.. Hence,  by dominated convergence theorem,   
	 		\begin{align}\label{st4}
	 			u(x)&=\E^{x}[\expo^{-\rho \tau^{\star}}u(X_{ \tau^{\star}})]-\E^{x}\Bigg[\int_{0}^{\tau^{\star}}\expo^{-\rho s}\Gamma u(X_{s-})\der s\Bigg].
	 		\end{align}
	 		If $x\geq b^{*}$, then Item (i) holds immediately. Otherwise, if $x<b^{\star}$, notice that $\Gamma u(X_{s-})=0$ for $s\in[0,\tau^{\star}]$ according to \eqref{st5}. Moreover, by the definition of $\tau^{*}$, it follows that $X_{\tau^{*}}\geq b^{\star}$. Thus, from these  and \eqref{st4}, we conclude that $u(x)=\E^{x}\big[\expo^{-\tau^{\star}\rho}[X_{\tau^{\star}}-c]\big]$ for any $x\in\R$. On the other hand, using \eqref{st3} and the fact that $\expo^{-\rho t}u(X_{t})-\int_{0}^{t}\expo^{-\rho s}\Gamma u(X_{s-})\der s$ is non-negative for $s\in[0,\tau^{\star}]$, we obtain $\E^{x}[\widetilde{M}_{t}]\geq-u(x)$.  Consequently,  by applying Fatou's Lemma, it follows that $\widetilde{M}_{t}$ is a supermartingale with $\E[\widetilde{M}_{0}]=0$. Hence $\E^{x}[\widetilde{M}_{\tau}] \leq0$ for any $\tau\in\mathscr{S}$. By setting  $t=\tau$ in \eqref{st3}, with $\tau\in\mathscr{S}$,  and then taking  expectation, we obtain $u(x)\geq\E^{x}[\expo^{-\rho \tau}u(X_{\tau})]\geq\E^{x}\big[\expo^{-\rho\tau}[X_{\tau}-c]\big]$ as indicated in \eqref{EDu1}. Therefore, by combining Items (i) and (ii), we conclude that $u=U$ on $\R$ and the stopping time $\tau^{*}=\tau(x)$ as in \eqref{tau1} is optimal for the problem \eqref{st1}.
	 		\end{proof}
	 
	 \section{Comparative statics analysis}\label{comp1}
	 
	 In this section, we study the sensitivity of the parameters $\mu$, $\sigma$, $\lambda_{\tmn}$, and $\lambda_{\tmp}$ on the value function $v_{b^{\star}}$ and the optimal threshold $b^{\star}$ as given by \eqref{v1.1} and \eqref{b}, respectively. To this end, we first present some properties of the roots of \eqref{eqaux}, which depend on these parameters.
	 
	 The following lemma is stated without proof, since it can be  verified by applying implicit differentiation to \eqref{eqaux} with respect to $\mu$, $\sigma$, $\lambda_{\tmn}$, respectively, and $\lambda_{\tmp}$, along with the use of \eqref{cond3}.
	 \begin{lema}\label{lemaraices}
	 	If $r$ is a root of \eqref{eqaux}, then it satisfies the following properties.
	 	\begin{enumerate}
	 		\item[(1)] The mapping $\sigma\mapsto r$, with $\sigma>0$, is increasing if $r$ is a negative root, and decreasing if $r$ is a positive root.
	 		
	 		\item[(2)] The mapping $\mu\mapsto r$, with $\mu\in\R$, is decreasing.
	 		
	 		\item[(3)] The mapping $\lambda_{\tmp}\mapsto r$, with $\lambda_{\tmp}>0$, is decreasing if $r$ is a negative root. If $r $ is a positive root of \eqref{eqaux}, the mapping $\lambda_{\tmp}\mapsto r$ is decreasing  when $r<\tilde{x}_{0}$, and increasing when  $r>\tilde{x}_{0}$, where $\tilde{x}_{0}>0$ is the unique point where 
	 		\begin{equation*}
	 			\frac{\rho}{x}+\lambda_{\tmn}\sum_{i=1}^{m_{\tmn}}\frac{\omega^{\tmn}_{i}}{\beta^{\tmn}_{i}+x}=\frac{\sigma x}{2}+\mu.
	 		\end{equation*}
	 		 \item[(4)] The mapping $\lambda_{\tmn}\mapsto r$, with $\lambda_{\tmn}>0$, is increasing if $r$ is a positive root. If $r $ is a  negative root of \eqref{eqaux}, the mapping $\lambda_{\tmn}\mapsto r$ is increasing  when $r>\tilde{x}_{1}$, and decreasing when  $r<\tilde{x}_{1}$, where $\tilde{x}_{1}<0$ is the unique point where 
	 		 \begin{equation*}
	 		 	\frac{\rho}{x}-\lambda_{\tmp}\sum_{i=1}^{m_{\tmp}}\frac{\omega^{\tmp}_{i}}{\beta^{\tmp}_{i}-x}=\frac{\sigma x}{2}+\mu.
	 		 \end{equation*}
	 	\end{enumerate}
	 \end{lema}
	 \begin{rem}\label{rem1}
	 	Since \eqref{eqaux} has $m_{\tmp}+1$ positive roots, it immediately follows from \eqref{des} and  Item (3) of Lemma \ref{lemaraices} that there exists a $k_{1}\in\{0,\dots,m_{\tmp}-1\}$ (independent of $\lambda_{\tmp}$) such  that  the mapping $\lambda_{\tmp}\mapsto r_{i}$ is decreasing on $(0,\infty)$ if $i\in\{0,\dots,k_{1}\}$, and is increasing on $(0,\infty)$ if $i\in\{k_{1}+1,\dots,m_{\tmp}\}$. Similarly, by \eqref{des} and Item (4) of Lemma \ref{lemaraices}, for the $m_{\tmn}+1$ negative roots, there exists a $k_{2}\in\{0,\dots,m_{\tmn}-1\}$ (independent of $\lambda_{\tmn}$) such the mapping $\lambda_{\tmn}\mapsto r^{\tmn}_{i}$ is increasing on $(0,\infty)$ if $i\in\{0,\dots,k_{2}\}$, and is decreasing on $(0,\infty)$ if $i\in\{k_{2}+1,\dots,m_{\tmn}\}$.
 	 \end{rem}
	 Taking derivatives of \eqref{b} with respect to $\mu$, $\sigma$, and $\lambda_{\tmn}$, respectively, and applying Lemma \ref{lemaraices} to these derivatives, the following corollary follows immediately; the proof is therefore omitted.
	 \begin{corol}
	 	Let $b^{\star}$ be as in \eqref{b}. Then, the following hold.
	 	\begin{enumerate}\label{b1}
	 		\item[(1)] The mappings $\sigma\mapsto b^{\star}$, with $\sigma>0$, and $\mu\mapsto b^{\star}$, with $\mu\in\R$, are increasing.	 		
	 		
	 		\item[(2)]  The mapping $\lambda_{\tmn}\mapsto b^{\star}$, with $\lambda_{\tmn}>0$, is decreasing.
	 		
	 	\end{enumerate}
	 \end{corol}
	 \begin{rem}
	 	Considering Remark \ref{rem1}, observe that \eqref{b} can be rewritten as $b^{*}=c-\sum_{i=2}^{m_{\tmp}+1}\frac{1}{\beta_{i-1}^{\tmp}}+\sum_{i=1}^{k_1+1}\frac{1}{r_{i-1}}+\sum_{i=k_{1}+2}^{m_{\tmp}+1}\frac{1}{r_{i-1}}$, where the last two sums are increasing and decreasing, respectively. The conjecture is that the mapping $\lambda_{\tmp}\mapsto b^{*}$, with $\lambda_{\tmp}>0$, is increasing as numerical examples support this claim. However, due to the dependency of $k_{1}$ on the parameters  $\mu$, $\sigma$,  $\lambda_{n}$, $\beta^{u}_{i}$, and $\omega^{u}_{i}$ with $u\in\{\tmn,\tmp\}$, it is not straightforward to rigorously verify this conjecture.
	 \end{rem}
	 
	The following result establishes some monotonicity properties of the  value  function $V$ with respect to  the parameters $\mu$, $\sigma$, $\lambda_{\tmn}$, and $\lambda_{\tmp}$.
	 \begin{prop}
	 	The value function $V$ given in \eqref{fV} is increasing with respect to $\mu$, $\sigma$ and $\lambda_{\tmn}$, and is decreasing with respect to $\lambda_{\tmp}$.
	 \end{prop}
	 \begin{proof}
	 	To verify that $V$ is increasing with respect to  $\mu$  and $\sigma$, respectively, one can apply arguments  similar  to those  in the proof of Proposition 5.2 of \cite{FK2021}. Therefore, we focus on showing the monotonicity of $V$ with respect to $\lambda_{\tmn}$ and $\lambda_{\tmp}$, respectively.
	 	
	 	Let us denote the value function $V$ and  the operator $\mathcal{D}$  given in \eqref{ENPV} and \eqref{generador}   by $ \mathcal{D}_{\lambda_{\tmp},\lambda_{\tmn}}$ and   $V_{\lambda_{\tmp},\lambda_{\tmn}}$, respectively, to explicitly show their dependence on the parameters $\lambda_{\tmp}$ and $\lambda_{\tmn}$.  Notice that for any $\lambda_{\tmn}>0$ and $\lambda_{\tmp}>0$,    \eqref{eqHJB} holds, and $V_{\lambda_{\tmp},\lambda_{\tmn}}(x,0)=0$ with $x\in\R$. Meanwhile,  if  $\lambda_{u,1}<\lambda_{u,2}$, with $u\in\left\{\tmn,\tmp\right\}$, we have that 
	 	\begin{align}\label{lambdapmon}
	 		&\mathcal{D}_{\sigma,\lambda_{\tmp,1},\lambda_{\tmn}}V_{\sigma,\lambda_{\tmp,2},\lambda_{\tmn}}(x,y)\notag\\
	 		&\quad=\mathcal{D}_{\sigma,\lambda_{\tmp,2},\lambda_{\tmn}}V_{\sigma,\lambda_{\tmp,2},\lambda_{\tmn}}(x,y)\notag\\
	 		&\qquad+\left(\lambda_{\tmp,1}-\lambda_{\tmp,2}\right)\int_{0}^{\infty}[V_{\sigma,\lambda_{\tmp,2},\lambda_{\tmn}}(x+z,y)-V_{\sigma,\lambda_{\tmp,2},\lambda_{\tmn}}(x,y)]\der P_{\tmn}(z)\notag\\
	 		&\quad\leq \left(\lambda_{\tmp,1}-\lambda_{\tmp,2}\right)\int_{0}^{\infty}[V_{\sigma,\lambda_{\tmp,2},\lambda_{\tmn}}(x+z,y)-V_{\sigma,\lambda_{\tmp,2},\lambda_{\tmn}}(x,y)]\der P_{\tmn}(z)\leq 0,
	 	\end{align}
	 	due to $\mathcal{D}_{\sigma,\lambda_{\tmp},\lambda_{\tmn}}V_{\sigma,\lambda_{\tmp},\lambda_{\tmn}}(x,y)\leq 0$  and the increasing property of the mapping  $x\mapsto V_{\sigma,\lambda_{\tmn},\lambda_{\tmn}}(x,y)$, for each $y\geq0$ fixed.  Similarly,  
	 	\begin{align}\label{lambdanmon}
	 		&\mathcal{D}_{\sigma,\lambda_{\tmp},\lambda_{\tmn,2}}V_{\sigma,\lambda_{\tmp},\lambda_{\tmn,1}}(x,y)\notag\\
	 		&\quad=\mathcal{D}_{\sigma,\lambda_{\tmp},\lambda_{\tmn,1}}V_{\sigma,\lambda_{\tmp},\lambda_{\tmn,1}}(x,y)\notag\\
	 		&\qquad+\left(\lambda_{\tmn,2}-\lambda_{\tmn,1}\right)\int_{0}^{\infty}[V_{\sigma,\lambda_{\tmp},\lambda_{\tmn,1}}(x-z,y)-V_{\sigma,\lambda_{\tmp},\lambda_{\tmn,1}}(x,y)]\der P_{\tmn}(z)\notag\\
	 		&\quad\leq \left(\lambda_{\tmn,2}-\lambda_{\tmn,1}\right)\int_{0}^{\infty}[V_{\sigma,\lambda_{\tmp},\lambda_{\tmn,1}}(x-z,y)-V_{\sigma,\lambda_{\tmp},\lambda_{\tmn,1}}(x,y)]\der P_{\tmn}(z)\leq 0.
	 	\end{align}
	 	By arguments similar to those in  the first part of Proposition \ref{TV}, and using  \eqref{lambdapmon} and \eqref{lambdanmon}, together with $-\alpha \frac{\partial}{\partial x}V_{\lambda_{\tmp},\lambda_{\tmn}}(x,y)-\frac{\partial}{\partial y}V_{\lambda_{\tmp},\lambda_{\tmn}}(x,y)+x-c\leq 0$ for any $\lambda_{\tmn}>0$ and $\lambda_{\tmp}>0$, we conclude that  $V_{\lambda_{\tmp,2},\lambda_{\tmn}}\leq V_{\lambda_{\tmp,1},\lambda_{\tmn}}$  and $V_{\lambda_{\tmp},\lambda_{\tmn,1}}\leq V_{\lambda_{\tmp},\lambda_{\tmn,2}}$.\qed
	 \end{proof}
	 
	 \begin{rem}
	 	Letting $\alpha \downarrow 0$ in \eqref{v1.1},  the region for selling a proportional amount of the commodity disappears. Moreover,   by  applying L'H\^opital's rule,  we obtain
	 	\begin{equation*}
	 		\lim\limits_{\alpha\to 0}V(x,y)=\begin{cases}
	 			y\sum_{j=1}^{m_{\tmp}+1}K_{j-1}\expo^{r_{j-1}[x-b^{\star}]} & \text{ si } x<b^{\star},\\
	 			[x-c]y & \text{ si } x\geq b^{\star},
	 		\end{cases}
	 	\end{equation*}
	 	which belongs to  $\hol^{1,1}(\R,\R^{+})$. Notice that $\parcial{^{2}V}{x^{2}}$ has a discontinuity at $x=b^{\star}$. Furthermore, letting $\alpha\uparrow\infty$  in \eqref{v1.1} yields that $V\rightarrow 0$.
	 	
	 \end{rem}
	 
	 \section{Conclusions and future work} \label{S8}

	 In this work, we addressed an optimal extraction problem where an agent operates in a spot market, and their actions have an additive, proportional, and negative effect on the commodity price. The commodity price dynamics, in the absence of agent interventions, are modelled by a drifted Brownian motion with jumps, capturing both continuous fluctuations and sudden market shocks.
	 
	 Our study showed that the optimal extraction strategy is of barrier-type: the agent’s decision to sell is based on comparing the current commodity price to a constant critical threshold $b^{\star}$. When the price falls below this level, it is optimal to refrain from selling any commodity. Conversely, by linearly comparing the current price and the inventory level, it may be optimal to sell either all available commodities or a portion of them, resulting in a decrease in the commodity price.
	 
	 The barrier strategy was derived from explicit solutions to the HJB equation associated with the value function, and its validity was justified by imposing smooth-fit conditions at the boundary. Notably, the threshold $b^{\star}$ is independent of both the current inventory level and the current commodity price; it depends solely on the initial parameters governing the price dynamics.
	 
	 On the other hand, in some commodity markets, empirical evidence suggests that prices exhibit mean reversion behavior, driven, for example, by supply and demand trends. Optimal extraction under mean-reverting price dynamics remains an open problem when jumps are included in the price without any intervention from the agent. Although this topic is beyond the scope of this article, it represents an important direction for future research.
	 	
	 Analytically, introducing mean reversion with a jump component significantly complicates the structure of the associated HJB equation. In particular, the resulting equation becomes an integro-differential equation, requiring more sophisticated analytical techniques   as viscosity solutions. Furthermore, establishing smooth-fit conditions and a verification theorem in this context is more intricate due to the non-linearity introduced by mean reversion.
	 	
	 	
	 	Based on the results obtained in \cite{FK2021} for the case where price dynamics are driven by an Ornstein-Uhlenbeck process, we intuitively expect that if the commodity price follows a mean-reverting process with a jump component, the structure of the optimal strategy will differ from the barrier-type solution found in our results. Specifically, the optimal strategy will be triggered by a non-constant critical level (curve) that should depend  on  the remaining inventory. The agent can no longer base their extraction decision solely on whether the current price exceeds a fixed threshold. Instead, the optimal strategy depends jointly on both the current price and the current reservoir level. When the price is high relative to its long-term mean, the optimal strategy might be to sell, anticipating a downward reversion. Conversely, if the price is low, the agent might prefer to postpone selling, expecting prices to recover toward the mean.

	\appendix
	
	\section{Appendix}
	
	\subsection{Proof of Lemma  \ref{Exp1}}\label{A}
	\begin{proof}[Proof of Lemma \ref{Exp1}]
		Let   $Y$ and $X$ be  as in \eqref{e1} and  \eqref{e2}, respectively, with   $\xi\in\mathcal{A}(y)$. Applying integration by parts and  It\^o-Meyer-Tanaka formula on $\{\expo^{-\rho t}f(X_{t},Y_{t}): t\geq0\}$, we get
		\begin{align}\label{ETV1}
			\expo^{-\rho t}f\left(X_{t},Y_{t}\right)-f(x,y)&=\int_{0}^{t}\expo^{-\rho s}(\mathcal{L}-\rho)f(X_s,Y_s)\der s
			+M_{t}\notag\\
			&\quad +\int_{0}^{t}\expo^{-\rho s}\left(-\alpha \frac{\partial}{\partial x}f(X_s,Y_s)-f_{y}(X_s,Y_s)\right)\der \xi_s^{c} \notag\\
			&\quad+\sum_{0\leq s\leq t}\expo^{-\rho s}(f(X_s,Y_s)-f(X_{s-},Y_{s-})),
		\end{align}
		where $\mathcal{L}=\frac{\sigma^{2}}{2}\parcial{^{2}}{x^{2}}+\mu\parcial{}{x}$ and $M_{t}\eqdef\int_{0}^{t}\sigma \expo^{-\rho s}\frac{\partial}{\partial x}f(X_{s})\der W_s$. Meanwhile, taking into account the random   Poisson measure  $\widetilde{N}^{u}$ on the space given in \eqref{e6.0}, notice that  
		\begin{align} \label{ETV2}
			&\sum_{0\leq s\leq t}\expo^{-\rho s}(f(X_s,Y_s)-f(X_{s-},Y_{s-})) \notag \\	
			&\quad=\sum_{0\leq s\leq t}\expo^{-\rho s}(f(X_{s-}-\Delta S^{\tmn}_{s}+\Delta S^{\tmp}_{s}-\alpha \Delta \xi_s,Y_{s-}-\Delta\xi_s)-f(X_{s-},Y_{s-}))\notag \\
			&\quad= \sum_{0\leq s\leq t}\expo^{-\rho s}(f(X_{s-}-\Delta S^{\tmn}_{s}+\Delta S^{\tmp}_{s}-\alpha \Delta \xi_s,Y_{s-}-\Delta\xi_s)- f(X_{s-}-\Delta S^{\tmn}_{s}+\Delta S^{\tmp}_{s},Y_{s-}))\notag\\
			&\qquad+ J_{t}^{\tmn}[X,Y] + J_{t}^{\tmp}[X,Y] +  \int_{0}^{t}\int_{0}^{\infty}\expo^{-\rho s}\big(f(X_{s-}-z,Y_{s-})-f(X_{s-},Y_{s-})\big)\lambda_{\tmn} {\der P}_{\tmn}(z)\der s \notag\\
			& \qquad+ \int_{0}^{t}\int_{0}^{\infty}\expo^{-\rho s}\left(f(X_{s-}+z,Y_{s-})-f(X_{s-},Y_{s-})\right)\lambda_{\tmp} {\der P}_{\tmp}(z)\der s,
		\end{align}
		with 
		\begin{align*}
			J_{t}^{\tmn}[X,Y] &\eqdef \int_{0}^{t}\int_{0}^{\infty}\expo^{-\rho s}\big(f(X_{s-}-z,Y_{s-})-f(X_{s-},Y_{s-})\big)(\widetilde{N}^{\tmn}(dz\times \der s)-\lambda_{\tmn} {\der P}_{\tmn}(z)\der s),\\
			J_{t}^{\tmp}[X,Y] &\eqdef\int_{0}^{t}\int_{0}^{\infty}\expo^{-\rho s}\big(f(X_{s-}+z,Y_{s-})-f(X_{s-},Y_{s-})\big)(\widetilde{N}^{\tmp}(dz\times \der s)-\lambda_{\tmp} {\der P}_{\tmp}(z)\der s).
		\end{align*}
		On the other hand, observe that
		\begin{align}
			\label{ETV3}
			&f(X_{s-}-\Delta S^{\tmn}_{s}+\Delta S^{\tmp}_{s}-\alpha \Delta \xi_s,Y_{s-}-\Delta\xi_s)-f(X_{s-}-\Delta S^{\tmn}_{s}+\Delta S^{\tmp}_{s},Y_{s-})\notag\\
			&\quad=\int_{0}^{\Delta \xi_s}\bigg(-\alpha \frac{\partial}{\partial x}f(X_{s-}-\Delta S^{\tmn}_{s}+\Delta S^{\tmp}_{s}-\alpha u,Y_{s-}-u)\notag\\
			&\quad\qquad\qquad\qquad\qquad-f_{y}(X_{s-}-\Delta S^{\tmn}_{s}+\Delta S^{\tmp}_{s}-\alpha u,Y_{s-}-u)\bigg)\der u,
		\end{align}
		Applying\eqref{ETV2}-\eqref{ETV3} in \eqref{ETV1}, we obtain
		\begin{align}
			\label{ETV10}
			&\expo^{-\rho t}f\left(X_{t},Y_{t}\right)-f(x,y) \notag\\
			&=\int_{0}^{t}\expo^{-\rho s}\mathcal{D}f(X_s,Y_s)\der s +{\int_{0}^{t}\sigma \expo^{-\rho s}\frac{\partial}{\partial x}f(X_{s})\der W_s}+ J_{t}^{\tmn}[X,Y]+J_{t}^{\tmp}[X,Y]\notag\\
			&\quad+\int_{0}^{t}\expo^{-\rho s}\left(-\alpha \frac{\partial}{\partial x}f(X_{s},Y_s)-f_{y}(X_{s},Y_s)\right)\der \xi_s^{c} \notag\\
			&\quad+\sum_{0\leq s\leq t}\expo^{-\rho s}   \int_{0}^{\Delta \xi_s}\bigg(-\alpha \frac{\partial}{\partial x}f(X_{s-}-\Delta S^{\tmn}_{s}+\Delta S^{\tmp}_{s}-\alpha u,Y_{s-}-u) \notag \\
			&\quad- f_{y}(X_{s-}-\Delta S^{\tmn}_{s}+\Delta S^{\tmp}_{s}-\alpha u,Y_{s-}-u)\bigg)\der u. 
		\end{align}
		Therefore, taking expected value in \eqref{ETV10} and considering that $M_{t}$, $J_{ t}^{\tmn}[X,Y] $ and $J_{ t}^{\tmp}[X,Y] $ are zero-mean-square-integrable martingales (since $f$ fulfils \eqref{l1}), we obtain   \eqref{ETV10.1}.\qed
	\end{proof}
	
	\subsection{Proof of Proposition \ref{pro1} and Lemma \ref{ru}}\label{B}
	
	In order to prove Proposition \ref{pro1} and Lemma \ref{ru}, we introduce some results and notation essential for these proofs. Take
	\begin{align}\label{co1}
		\begin{split}
			&\bar{\eta}_{j}\eqdef  \prod_{k=2}^{j}[r_{j-1}-\beta^{ {\tmp}}_{k-1}], \quad\underbar{$\eta$}^{n}_{j}\eqdef\prod_{k=j+1}^{n+1}[\beta^{ {\tmp}}_{k-1}-r_{j-1}],\\
			&\bar{\gamma}_{j}\eqdef  \prod_{k=1}^{j-1}[r_{j-1}-r_{k-1}], \quad\underbar{$\gamma$}^{n}_{j}\eqdef\prod_{k=j+1}^{n+1}[r_{k-1}-r_{j-1}]\\
			&\underbar{$\nu$}_{i}\eqdef\prod_{k= {1}}^{i-1}[\beta^{\tmp}_{i-1}-r_{k-1}],\quad \bar{\nu}^{n}_{i}\eqdef\prod_{k=i}^{n+1}[r_{k-1}-\beta^{\tmp}_{i-1}],\\ 
			&\underbar{$\theta$}_{i}\eqdef\prod_{k=2}^{i-1}[\beta^{\tmp}_{i-1}-\beta^{\tmp}_{k-1}],\quad \bar{\theta}^{n}_{i}\eqdef\prod_{k=i+1}^{n+1}[\beta^{\tmp}_{k-1}-\beta^{\tmp}_{i-1}].
		\end{split}
	\end{align} 
	We set $\bar{\eta}_{j}$ and $\bar{\gamma}_{j}$ to be one  if $j=1$,  {and $\underbar{$\eta$}^{n}_{j}$ and $\underbar{$\gamma$}^{n}_{j}$ to be one if $n=1$ and $j=2$}.  Additionally, we  take  $\underbar{$\theta$}_{i}$ identically equal to one if $i=2$,  {and   $\bar{\theta}^{n}_{i}$ to be one if $n=1$ and $i=2$}. Due to \eqref{des}, $\underbar{$\nu$}_{i}$, $\bar{\nu}^{n}_{i}$, $\underbar{$\theta$}_{i}$, and $\bar{\theta}^{n}_{i}$ are positive. The proofs of the following complementary  results can be  found in Appendix \ref{C}
	\begin{lema}\label{p2A2}
		For each $j\in\{1,\dots,n+1\}$ fixed,
		\begin{equation}\label{cof1}
			\frac{\Cof_{1, j}[A_n]}{\det[A_n]}=\frac{\bar{\eta}_{j}\,\underbar{$\eta$}^{n}_{j}}{\bar{\gamma}_{j}\,\underbar{$\gamma$}^{n}_{j}},
		\end{equation}
		Furthermore, for each $i\in\{2,\dots,n+1\}$ fixed,
		\begin{equation}\label{cof2}
			\sum_{k=1}^{n+1}\frac{{\beta^{\tmn}_{i-1}}\Cof_{1,k}[A_{n}]}{\beta^{\tmp}_{i-1}-r_{k-1}}=0.
		\end{equation}
	\end{lema}
	The following corollary is a direct result of Lemma \ref{p2A2}, therefore its proof will be omitted.   
	\begin{corol}
			For each $j\in\{1,\dots,n+1\}$ fixed, let $A^{j}_{n}[I_{n+1}^{\trans}]$ be the matrix formed by replacing the $j$-th column of $A_{n}$ by $I^{\trans}_{n+1}$, where $I_{n+1}$ is as in \eqref{ma1}.  Then,
		\begin{equation}\label{cof3}
			\det [A^{j}_{n}[I_{n+1}^{\trans}]]=\frac{\mathcal{R}^{n}}{r_{j-1}}\Cof_{1, j}[A_n],
		\end{equation}
		where $\mathcal{R}^{n}$ is defined as in \eqref{Cof3}.
	\end{corol}
	
	\begin{lema}\label{p2A3}
		For each $(i,j)\in\{2,\dots,n+1\}\times\{1,\dots,n+1\}$ fixed,
		\begin{align}
			&\Cof_{i,j}[A_{n}]=\frac{\mathcal{M}^{n}_{i-1}}{\beta^{\tmp}_{i-1}-r_{j-1}}\Cof_{1,j}[A_{n}],\label{Cof2}\\
			&\mathcal{R}^{n}=r_{j-1}\Bigg[1+\sum_{i=2}^{n+1}\frac{\mathcal{M}_{i-1}^{n}}{\beta_{i-1}^{\tmp}-r_{j-1}}\Bigg],\label{Rn}\\
			&\frac{1}{\mathcal{R}^{n}}\left[1+\sum_{i=2}^{n+1}\frac{\mathcal{M}_{i-1}^{n}}{\beta_{i-1}^{\tmp}}\right]=\sum_{i=1}^{n+1}\frac{1}{r_{i-1}}-\sum_{i=2}^{n+1}\frac{1}{\beta_{i-1}^{\tmp}}.\label{Cof5}
		\end{align}
		where  $\mathcal{R}^{n}$  is given in  \eqref{Cof3} and  $\mathcal{M}^{n}_{{i-1}}$ is defined as follows
		\begin{equation}\label{Cof3.1}
			\mathcal{M}^{n}_{{i-1}}=\frac{\underbar{$\nu$}_{i}\,\bar{\nu}^{n}_{i}}{\beta^{\tmp}_{i-1}\,\underbar{$\theta$}_{i}\,\bar{\theta}^{n}_{i}}.
		\end{equation}
	\end{lema}
	
				
	
	\subsubsection{Proof of Proposition \ref{pro1}}
	\begin{proof}[Proof of Proposition \ref{pro1}]
		Let $(x,y)$ belong in $\mathcal{W}$. By applying \eqref{cond3.0},  \eqref{generador}, \eqref{cond1}, and  \eqref{cond2} in \eqref{cond2.1}, we deduce the following  partial integro-differential equation
		
		\begin{align}\label{eqfund}
			&\frac{\sigma^{2}}{2}\parcial{^{2}}{x^{2}}v_{b^{\star}}(x,y)+\mu\parcial{}{x}v_{b^{\star}}(x,y)-[\rho+\lambda_{\tmp}+\lambda_{\tmp}]v_{b^{\star}}(x,y)\notag\\
			&\quad+\lambda_{\tmn}\sum_{k=2}^{m_{\tmn}+1}\beta_{k-1}^{\tmn}\omega^{\tmn}_{k-1}\expo^{-\beta_{k-1}^{\tmn}x}\int_{-\infty}^{x}v_{b^{\star}}(z,y)\expo^{\beta_{k-1}^{\tmn}z}\der z\notag\\
			&\quad+\lambda_{\tmp}\sum_{k=2}^{m_{\tmp}+2}\beta_{k-1}^{\tmp}\omega^{\tmp}_{k-1}\expo^{\beta_{k-1}^{\tmp}x}\bigg\{\int_{x}^{b^{\star}}v_{b^{\star}}(z,y)\expo^{-\beta_{k-1}^{\tmp}z}\der z +C_{k-1}(y;b^{\star})\bigg\}=0,
		\end{align}
		where $C_{k-1}(y; b^{\star})$ is defined as
		\begin{align}\label{eqfund0.1}
			&C_{k-1}(y;b^{\star})\notag\\
			&=\int_{b^{\star}}^{\alpha y+b^{\star}}v_{b^{\star}}(b^{\star},y-\mathbb{Y}_{b^{\star}}(z))\expo^{-\beta_{k-1}^{\tmp}z}\der z\notag\\
			&\quad+\int_{b^{\star}}^{\alpha y+b^{\star}}\bigg[\frac{[z-c]^{2}-[c-b^{\star}]^{2}}{2\alpha}\bigg]\expo^{-\beta_{k-1}^{\tmp}z}\der z+\int_{\alpha y+b^{\star}}^{\infty}\bigg[[z-c]y-\alpha\frac{y^{2}}{2}\bigg]\expo^{-\beta_{k-1}^{\tmp}z}\der z\notag\\
			&=\int_{b^{\star}}^{\alpha y+b^{\star}}v_{b^{\star}}(b^{\star},y-\mathbb{Y}_{b^{\star}}(z))\expo^{-\beta_{k-1}^{\tmp}z}\der z+\frac{[1+\beta_{k-1}^{\tmp}[b^{\star}-c]] [1-\expo^{-\alpha  \beta_{k-1}^{\tmp}   y}] \expo^{-\beta_{k-1}^{\tmp} b^{\star}}}{\alpha  [\beta_{k-1}^{\tmp} ]^{3}}.
		\end{align} 	
		By applying the differential operator $\prod_{j=2}^{m_{\tmn}+1}\left(\parcial{}{x}+\beta_{j-1}^{\tmn}\right)\prod_{i=2}^{m_{\tmp}+1}\left(\parcial{}{x}-\beta_{i-1}^{\tmp}\right)$ on the integro-dif\-fe\-ren\-tial equation \eqref{eqfund}, we determine that $v_{b^{\star}}$ must satisfy a linear homogenous differential equation with constant coefficients of order $m_{\tmp}+m_{\tmn}+2$, whose auxiliary equation is provided by \eqref{eqaux}. Utilizing Theorem 3.2 in \cite{CK2011},  we have that 
		$$v_{b^{\star}}(x,y)=\sum_{i=1}^{m_{\tmp}+1}A^{\tmp}_{i-1}(y)\expo^{r_{i-1} x}+\sum_{j=1}^{m_{\tmn}+1}A^{\tmn}_{j-1}(y)\expo^{r^{\tmn}_{j-1} x}$$ 
		is a solution of \eqref{eqfund} for $(x,y)\in\mathcal{W}$, where $A^{\tmp}_{i-1}$ and $A^{\tmn}_{j-1}$ are functions that do not depend on   $x$. However, since the solution to the HJB equation must satisfies  \eqref{l1}, it follows that $A^{\tmn}_{j-1}$ must be zero for $j \in \{1, \dots, m_{\tmn}+1\}$.   Then,
		\begin{equation}\label{v1}
			v_{b^{\star}}(x,y)=\sum_{j=1}^{m_{\tmp}+1}A^{\tmp}_{j-1}(y)\expo^{r_{j-1} x}\quad \text{for}\ (x,y)\in\mathcal{W}.
		\end{equation} 
		Since we look for a solution that is continuous differentiable at the boundary $x=b^{\star}$, by \eqref{v1} and \eqref{v2}, we get 
		$\sum_{j=1}^{m_{\tmp}+1}\expo^{r_{j-1} b^{\star}}[\alpha r_{j-1}A^{\tmp}_{j-1}(y)+A_{j-1}^{\tmp\,\prime}(y)]=b^{\star}-c$.
		This equation implies the existence of constants $K_{j-1}$ (with $j\in\{1, \dots, m_{\tmp}+1\}$) such that  \eqref{cond1K} holds, and
		\begin{equation}\label{Ai}
			\alpha r_{j-1}A^{\tmp}_{j-1}(y)+A_{j-1}^{\tmp\,\prime}(y)=K_{j-1}\expo^{-r_{j-1} b^{\star}}\quad \text{for}\  j\in\{1,\dots ,m_{\tmp}+1\}.
		\end{equation}
		Solving \eqref{Ai}, we obtain $A^{\tmp}_{j-1}(y)=\frac{K_{j-1}}{\alpha r_{j-1}}\left(1- {\ell_{j-1}}\expo^{-\alpha r_{j-1} y}\right)\expo^{-r_{j-1}b^{\star}}$.
		Since 
		$$\sum_{j=1}^{m_{\tmp}+1}\frac{K_{j-1}}{\alpha r_{j-1}}(1- {\ell_{j-1}})\expo^{-r_{j-1 }b^{\star}}\expo^{r_{j-1}x}=v_{b^{\star}}(x,0)=0\quad \text{for all $x\in\R$,}$$ 
		we conclude that $\ell_{j-1}=1$ for all $j\in\{1,\dots,m_{\tmp}+1\}$. Consequently,
		\begin{equation}\label{v3}
			A^{\tmp}_{j-1}(y)=\frac{K_{j-1}}{\alpha r_{j-1}}\left(1-\expo^{-\alpha r_{j-1} y}\right)\expo^{-r_{j-1}b^{\star}}.
		\end{equation}
		From \eqref{cond1}--\eqref{cond2.1}, \eqref{v1} and \eqref{v3}, we derive that $v_{b^{\star}}$ is given by \eqref{v1.1}. Now, through   a smooth fit analysis, we shall establish that  $b^{\star}$  and $K_{i}$ are as in \eqref{b} and \eqref{K2}, respectively. By differentiating  \eqref{v2} w.r.t. $x$,  we observe $\alpha\frac{\partial^2}{\partial x^{2}}v_{b^{\star}}(x,y)+\frac{\partial^2}{\partial x\partial y}v_{b^{\star}}(x,y)=1$ for $(x,y)\in\mathcal{S}_{1}$.
		For $v_{b^{\star}}$ to have continuous second derivatives at the boundary $x = b^{\star}$, it must satisfy
		$\lim_{x\uparrow b^{\star}}\bigg[\alpha\frac{\partial^2}{\partial x^{2}}v_{b^{\star}}(x,y)+\frac{\partial^2}{\partial x\partial y}v_{b^{\star}}(x,y)\bigg]=1$.
		Hence, based on this and utilizing \eqref{v1.1}, the parameters $K_{j-1}$ must satisfy the condition \eqref{cond2K}. Furthermore, considering \eqref{v1.1} again, for each $k \in {2, \ldots, m_{\tmp}+1}$,
		\begin{align*}
			&\int_{-\infty}^{x}v_{b^{\star}}(z,y)\expo^{\beta_{k-1}^{\tmn}z}\der z=\sum_{j=1}^{m_{\tmp}+1}\frac{A^{\tmp}_{j-1}(y)}{r_{j-1}+\beta_{k-1}^{\tmp}}\expo^{[r_{j-1}+\beta_{k-1}^{\tmp}]x},\\
			&\int_{x}^{b^{\star}}v_{b^{\star}}(z,y)\expo^{-\beta_{k-1}^{\tmp}z}\der z=\sum_{j=1}^{m_{\tmp}+1}\frac{A^{\tmp}_{j-1}(y)}{r_{j-1} -\beta_{k-1}^{\tmp}}\expo^{\beta_{k-1}^{\tmp}x}[\expo^{[r_{j-1} -\beta_{k-1}^{\tmp}]b^{\star}}-\expo^{[r_{j-1} -\beta_{k-1}^{\tmp}]x}],\\
			&\int_{b^{\star}}^{\alpha y+b^{\star}}v_{b^{\star}}(b^{\star},y-\mathbb{Y}_{b^{\star}}(z))\expo^{-\beta_{k-1}^{\tmp}z}\der z\notag\\
			&\quad=\sum_{j=1}^{m_{\tmp}+1}\frac{K_{j-1}\expo^{-\beta_{k-1}^{\tmp}b^{\star}}\Big[r_{j-1}[1-\expo^{-\beta_{k-1}^{\tmp}\alpha  y}]-\beta_{k-1}^{\tmp}[ 1-\expo^{-\alpha r_{j-1} y}]\Big] }{\alpha  \beta_{k-1}^{\tmp}  r_{j-1} (r_{j-1}-\beta_{k-1}^{\tmp} )}.
		\end{align*}	
		When applying the equations above  to \eqref{eqfund}--\eqref{eqfund0.1}, it can be verified that  for $(x,y)\in\mathcal{W}$, the following equation hods
		\begin{align}\label{eqfund2}
			&\sum_{j=1}^{m_{\tmp}+1}A^{\tmp}_{j-1}(y)\expo^{r_{j-1} x}\bigg[\frac{\sigma^{2}}{2}r^{2}_{j-1}+\mu r_{j-1}-[\rho+\lambda_{\tmp}+\lambda_{\tmn}]\notag\\
			&\quad+\lambda_{\tmn}\sum_{k=2}^{m_{\tmn}+1}\frac{\beta_{k-1}^{\tmn}\omega^{\tmn}_{k-1}}{r_{j-1}+\beta_{k-1}^{\tmn}}+\lambda_{\tmp}\sum_{k=2}^{m_{\tmp}+1}\frac{\beta_{k-1}^{\tmp}\omega^{\tmp}_{k-1}}{\beta_{k-1}^{\tmp}-r_{j-1}}\bigg]\notag\\
			&\quad +\frac{\lambda_{\tmp}}{\alpha}\sum_{k=2}^{m_{\tmp}+2}\omega^{\tmp}_{k-1}\expo^{\beta_{k-1}^{\tmp}[x-b^{\star}]}[1-\expo^{-\beta_{k-1}^{\tmp}\alpha  y}]\bigg[\sum_{j=1}^{m_{\tmp}+1}\frac{K_{j-1}}{r_{j-1} -\beta_{k-1}^{\tmp}}+\frac{[1+\beta_{k-1}^{\tmp}[b^{\star}-c]] }{ [\beta_{k-1}^{\tmp} ]^{2}}\bigg]=0.
		\end{align}
		Using \eqref{eqaux}, it follows that  \eqref{eqfund2}  holds  true if and only if for each $k\in\{2,\dots,m_{\tmp}+1\}$,  the  condition \eqref{cond3K} is met.
		Based on  the previously analysis, we get that  the constants $K_{j-1}$ and $b^{\star}$ must satisfy the system of equations  $\eqref{cond1K}$--$\eqref{cond3K}$. First, we    express $K_{j-1}$ in terms   of $b^{\star}$ by solving the quadratic system of equations  given by \eqref{cond1K} and \eqref{cond3K},  which is of size $m_{\tmp}+1$. This system of equations can be represented in  matrix form as  $A_{m_{\tmp}} K-C=0$,
		where $A_{m_{\tmp}}$ is as in \eqref{ma1}, $K\eqdef(K_{0},K_{1},\dots,K_{m_{\tmp}+1})^{\trans}$, and $C=\bigg(b^{\star}-c,b^{\star}-c+\frac{1}{\beta^{\tmp}_{1}},b^{\star}-c+\frac{1}{\beta^{\tmp}_{2}},\dots,b^{\star}-c+\frac{1}{\beta^{\tmp}_{m_{\tmp}}}\bigg)^{\trans}$. The notation $M^{\trans}$ signifies the transpose of any matrix $M$. By Cramer's rule, we obtain $K_{j-1}=\frac{\det[A^{j}_{m_{\tmp}}[C]]}{\det [A_{m_{\tmp}}]}$ for  $j\in\{1,\dots,m_{\tmp}+1\}$, 
		where $A^{j}_{m_{\tmp}}[C]$ is the matrix formed by replacing the $j$-th column of $A_{m_{\tmp}}$ by the column vector $C$.  Noticing that $C=[b^{\star}-c]I_{m_{\tmp}+1}^{\trans}+D$, with $D=\Big(0,\frac{1}{\beta^{\tmp}_{1}},\frac{1}{\beta^{\tmp}_{2}},\dots,\frac{1}{\beta^{\tmp}_{m_{\tmp}}}\Big)^{\trans}$, we have for $j\in\{1,\dots,m_{\tmp}+1\}$,
		\begin{align}\label{l2}
			K_{j-1}&=\frac{1}{\det [A_{m_{\tmp}}]}\{[b^{\star}-c]\det[A^{j}_{m_{p}}[I_{m_{\tmp+1}}^{\trans}]]+\det[A^{j}_{m_{p}}[D]]\}.
		\end{align}
		Using \eqref{Cof2}, notice that 
		\begin{align}\label{l3}
			\det[A^{j}_{m_{p}}[D]]=\sum_{k=2}^{m_{\tmp}+1}\frac{1}{\beta^{\tmp}_{k-1}}\Cof_{k,j}[A_{m_{\tmp}}]=\Cof_{1,j}[A_{m_{\tmp}}]\sum_{k=2}^{m_{\tmp}+1}\frac{\mathcal{M}_{k-1}^{m_{\tmp}}}{\beta^{\tmp}_{k-1}[\beta^{\tmp}_{k-1}-r_{j-1}]}.
		\end{align}
		Applying  \eqref{l3} in \eqref{l2} an considering \eqref{cof3}, we see that
		\begin{equation}\label{K1}
			K_{j-1}=\left[(b^{\star}-c)\frac{\mathcal{R}^{m_{\tmp}}}{r_{j-1}}+\sum_{k=2}^{{m_{\tmp}+1}}\frac{\mathcal{M}_{k-1}^{m_{\tmp}}}{\beta_{k-1}^{\tmp}[\beta_{k-1}^{\tmp}-r_{j-1}]}\right]\frac{\Cof_{1,j}[A_{m_{\tmp}}]}{\det [A_{m_{\tmp}}]}.
		\end{equation}
		Then, substituting \eqref{K1} into \eqref{cond2K} and using  the fact that 
		$\det [ A_{m_{\tmp}}]=\sum_{k=1}^{{m_{\tmp}}+1}\Cof_{1,k}[A_{m_{\tmp}}]$
		and 
		\begin{equation}\label{K1.1}
			-\frac{r_{j-1}}{\beta_{k-1}^{\tmp}-r_{j-1}}=1-\frac{\beta_{k-1}^{\tmp}}{\beta_{k-1}^{\tmp}-r_{j-1}}, 
		\end{equation}
		we obtain
		\begin{align*}
			1&=\sum_{j=1}^{{m_{\tmp}}+1} r_{j-1}\left[\left((b^{\star}-c)\frac{\mathcal{R}^{m_{\tmp}}}{r_{j-1}}+\sum_{k=2}^{{m_{\tmp}}+1}\frac{\mathcal{M}_{k-1}^{m_{\tmp}}}{\beta_{k-1}^{\tmp}(\beta_{k-1}^{\tmp}-r_{j-1})}\right)\frac{\Cof_{1,j}[A_{m_{\tmp}}]}{\det [A_{m_{\tmp}}]}\right]\\
			&=(b^{\star}-c)\mathcal{R}^{m_{\tmp}}-\sum_{k=2}^{{m_{\tmp}}+1}\frac{\mathcal{M}_{k-1}^{m_{\tmp}}}{\beta_{k-1}^{\tmp}\det [A_{m_{\tmp}}]}\left(\det[A_{m_{\tmp}}]-\beta_{k-1}^{\tmp}\sum_{j=1}^{{m_{\tmp}}+1}\frac{\Cof_{1,j}[A_{m_{\tmp}}]}{\beta_{k-1}^{\tmp}-r_{j-1}}\right)
			.		
		\end{align*}
		By \eqref{cof2}, we deduce that 
		\begin{equation}\label{b2}
		b^{\star}=c+\frac{1}{\mathcal{R}^{m_{\tmp}}}\bigg[1+\sum_{k=2}^{m_{\tmp}+1}\frac{\mathcal{M}_{k-1}^{m_{\tmp}}}{\beta_{k-1}^{\tmp}}\bigg].
		\end{equation}
		 Consequently, substituting \eqref{b2} in \eqref{K1} and utilizing \eqref{K1.1},  we find that $K_{j-1}$ is given in \eqref{K2}. Moreover, by \eqref{Cof3} and \eqref{cof1}, we conclude  that  $b^{\star}$ and $K_{j-1}$ are positive. Finally, $b^{\star}$ is given as in \eqref{b}, due to \eqref{Cof5}.\qed
	\end{proof}
	
	\subsubsection{Proof of Lemma \ref{ru}}
	\begin{proof}[Proof of Lemma \ref{ru}]
		By \eqref{eqaux}, observe that $K_{j-1}\,p(r_{j-1})=0$ for $j\in\{1,\dots,m_{\tmp}+1\}$. This implies that
		\begin{multline*}
			\frac{\sigma^{2}}{2}\sum_{j=1}^{m_{\tmp}+1}r_{j-1}^{2}K_{j-1}+\mu\sum_{j=1}^{m_{\tmp}+1}r_{j-1} K_{j-1}-(\rho+\lambda_{\tmn}+\lambda_{\tmp})\sum_{j=1}^{m_{\tmp}+1}K_{j-1}\\
			+\lambda_{\tmn}\sum_{k=2}^{m_{\tmn}+1}\sum_{j=1}^{m_{\tmp}+1}\frac{\omega_{k-1}^{\tmn}\beta_{k-1}^{\tmn}}{\beta_{k-1}^{\tmn}+r_{j-1}}K_{j-1}+\lambda_{\tmp}\sum_{k=2}^{m_{\tmn}+1}\omega_{k-1}^{\tmp}\sum_{j=1}^{m_{\tmp}+1}\frac{\beta_{k-1}^{\tmp}}{\beta_{k-1}^{\tmp}-r_{j-1}}K_{j-1}=0.
		\end{multline*}
		Using \eqref{K2} and considering that $\det[A_{m_{\tmp}}]=\sum_{j=1}^{m_{\tmp}+1}\Cof_{1,j}[A_{m_{\tmp}}]$,  it is straightforward to see that $\sum_{j=1}^{m_{\tmp}+1}r_{j-1}^{2}K_{j-1}=\mathcal{R}^{m_{\tmp}}$. Applying \eqref{cond1K}--\eqref{cond3K} to the  equation above, we get
		\begin{multline*}
			\frac{\sigma^{2}}{2}\mathcal{R}^{m_{\tmp}}+\mu-(\rho+\lambda_{\tmn}+\lambda_{\tmp})(b^{\star}-c)\\
			+\lambda_{\tmn}\sum_{k=2}^{m_{\tmn}+1}\sum_{j=1}^{m_{\tmp}+1}\frac{\omega_{k-1}^{\tmn}\beta_{k-1}^{\tmn}}{\beta_{k-1}^{\tmn}+r_{j-1}}K_{j-1}+
			\lambda_{\tmp}\sum_{k=2}^{m_{\tmn}}\omega_{k-1}^{\tmp}\left(b^{\star}-c+\frac{1}{\beta_{k-1}^{\tmp}}\right)=0.
		\end{multline*}
		Consequently, it follows that   \eqref{equ1} holds true, as per \eqref{cond3}. Using  \eqref{eqaux} again, we have that $\frac{K_{j-1}}{r_{j-1}}p(r_{j-1})=0$. Then, from here and taking into account \eqref{K1.1} and 
		\begin{equation}\label{K3}
			\frac{r_{j-1}}{\beta_{k-1}^{\tmn}+r_{j-1}}=1-\frac{\beta_{k-1}^{\tmn}}{\beta_{k-1}^{\tmn}+r_{j-1}},
		\end{equation}
		it follows that
		\begin{multline*}
			\frac{\sigma^{2}}{2}\sum_{j=1}^{m_{\tmp}+1}r_{j-1}K_{j-1}+\mu \sum_{j=1}^{m_{\tmp}+1}K_{j-1}-\rho\sum_{j=1}^{m_{\tmp}+1}\frac{K_{j-1}}{r_{j-1}}\\
			-\lambda_{\tmn}\sum_{k=2}^{m_{\tmn}+1}\sum_{j=1}^{m_{\tmp}+1}\frac{\omega_{k-1}^{\tmn}}{\beta_{k-1}^{\tmn}+r_{j-1}}K_{j-1}+\lambda_{\tmp}\sum_{k=2}^{m_{\tmp}+1}\sum_{j=1}^{m_{\tmp}+1}\frac{\omega_{k-1}^{\tmp}}{\beta_{k-1}^{\tmp}-r_{j-1}}K_{j-1}=0.
		\end{multline*}
		Hence, from \eqref{cond1K} and \eqref{cond2K} we obtain \eqref{equ2}. Finally, by \eqref{cond1K} and \eqref{cond2K}, $$\Xi^{\tmn}_{k-1}=\sum_{j=1}^{m_{\tmp}+1}\left(\frac{\beta_{k-1}^{\tmn}}{\beta_{k-1}^{\tmn}+r_{j-1}}-1+\frac{r_{j-1}}{\beta_{k-1}^{\tmn}}\right)K_{j-1}.$$
		From here and \eqref{K3}, we obtain the right-hand side of \eqref{equ3}.\qed
	\end{proof}
	
	\subsection{Proofs of some technical results discussed in Subsection \ref{B} }\label{C}
	\begin{proof}[Proof of Lemma \ref{p2A2}]
		Let $j\in\{1,\dots,n+1\}$ be fixed. By subtracting the column $j$ of $A_{n}$ from each of its  other  columns $j'\in\{1,\dots,n+1\}\setminus\{j\}$,  we obtain a matrix $C=(c_{i,j'})$ with entries given by
		\begin{align*}
			c_{i,j'}
			&=\begin{cases}
				0&\text{if}\ i=1\ \text{and}\ j'\in\{1,\dots,n+1\}\setminus\{j\},\\
				1& \text{if}\ i=1\ \text{and}\ j'=j,\\
				\frac{\beta^{\tmp}_{i-1}}{\beta^{\tmp}_{i-1}-r_{j-1}}& \text{if}\ i\in\{2,\dots,n+1\}\ \text{and}\ j'=j,\\
				\frac{\beta^{\tmp}_{i-1}[r_{j'-1}-r_{j-1}]}{[\beta^{\tmp}_{i-1}-r_{j-1}][\beta^{\tmp}_{i-1}-r_{j'-1}]}& \text{if}\ i\in\{2,\dots,n+1\}\ \text{and}\ j'\in\{1,\dots,n+1\}\setminus\{j\},
			\end{cases}
		\end{align*}
				that satisfies $\det[A_n]=\det[C] $. By extracting $r_{j'-1}-r_{j-1}$ from  each column $j'\in\{1,\dots,n+1\}\setminus\{j\}$,  and then extracting $\frac{1}{\beta_{i-1}-r_{j-1}}$ from each row $i\in\{2,\dots,n+1\}$ we have  that
				\begin{align}\label{m1}
					\det[A_n]&={\frac{\varrho^{1}_{j}}{ \varrho^{2}_{j}}}\det\left(\begin{matrix}
						0 & \cdots & 0 & 1 & 0 & \cdots  & 0\\
						\frac{\beta^{\tmp}_1}{\beta^{\tmp}_1-r_{0}} & \cdots& 	\frac{\beta^{\tmp}_1}{\beta^{\tmp}_1-r_{j-2}} & 	\beta^{\tmp}_1 & 	\frac{\beta^{\tmp}_1}{\beta^{\tmp}_1-r_{j}} & \cdots & 	\frac{\beta^{\tmp}_1}{\beta^{\tmp}_1-r_{n}} \\
						\vdots & \ddots & \vdots & \vdots &\vdots & \ddots & \vdots \\
						\frac{\beta^{\tmp}_n}{\beta^{\tmp}_n-r_{0}} & \cdots& 	\frac{\beta^{\tmp}_n}{\beta^{\tmp}_n-r_{j-2}} & \beta^{\tmp}_n & 	\frac{\beta^{\tmp}_n}{\beta^{\tmp}_n-r_{j}} & \cdots & 	\frac{\beta^{\tmp}_n}{\beta^{\tmp}_n-r_{n}} 
					\end{matrix}\right)
					=\frac{\varrho^{1}_{j}}{ \varrho^{2}_{j}}\Cof_{1,j}[ A_{n}],
				\end{align}
				with $\varrho^{1}_{j}\eqdef \prod_{\substack{k=1\\
							k\neq j}}^{n+1}(r_{k-1}-r_{j-1})$ and $\varrho^{2}_{j}\eqdef \prod_{k=2}^{n+1}(\beta^{\tmp}_{k-1}-r_{j-1})$. Observe that 
				\begin{align}\label{m2}
					\varrho^{1}_{j}=[-1]^{j-1} {\bar{\gamma}_{j}\, \underbar{$\gamma$}_{j}}\quad \text{and}\quad
					\varrho^{2}_{j}=[-1]^{j-1} {\bar{\eta}_{j}\,\underbar{$\eta$}_{j}},
				\end{align}
				because of \eqref{des}. Recall that $\bar{\gamma}_{j}$, $\underbar{$\gamma$}_{j}$, $\bar{\eta}_{j}$, and $\underbar{$\eta$}_{j}$ are defined in \eqref{co1}. Therefore,  by applying \eqref{m2} in \eqref{m1}, we obtain \eqref{cof1}. It can be easily verified using \eqref{ma1} that $\sum_{k=1}^{n+1}\frac{{\beta^{\tmn}_{i-1}}\Cof_{1,k}[A_{n}]}{\beta^{\tmp}_{i-1}-r_{k-1}}$ equals  the determinant of a  matrix whose first and $i$-th raw are identical, indicating that this determinant is zero,  thus  \eqref{cof2} holds. 
			\end{proof}
			
			\begin{proof}[Proof of Lemma \ref{p2A3}]
				 {For $n=1$, by \eqref{ma1}, it is straightforward to check the validity of \eqref{Cof2}. Let $n\geq2$ and}  fix $(i,j)\in\{2,\dots,n+1\}\times\{1,\dots,n+1\}$. By subtracting the $i$-th row  of $A_{n}$ from each of its  other  rows $i'\in\{1,\dots,n+1\}\setminus\{i\}$,  we obtain a matrix $C^{1}=(c^{1}_{i',j'})$ with entries given by
				\begin{align*}
					&c^{1}_{i',j'}=
					\begin{cases}
						\frac{-r_{j'-1}}{\beta^{\tmp}_{i-1}-r_{j'-1}}&\text{if}\  i'=1\ \text{and}\ j'\in\{1,\dots,n+1\},\\
						\frac{\beta^{\tmp}_{i-1}}{\beta^{\tmp}_{i-1}-r_{j'-1}}&\text{if}\ i'=i\ \text{and}\ j'\in\{1,\dots,n+1\},\\
						\frac{r_{j'-1}[\beta^{\tmp}_{i-1}-\beta^{\tmp}_{i'-1}]}{[\beta^{\tmp}_{i'-1}-r_{j'-1}][\beta^{\tmp}_{i-1}-r_{j'-1}]} &\text{if}\ i'\in\{2,\dots,n+1\}\setminus\{i\}\\ &\text{and}\ j'\in\{1,\dots,n+1\},\\
					\end{cases}
				\end{align*}
						that satisfies $\det[A_n]=\det[C^1] $. By extracting $\beta^{\tmp}_{i-1}-\beta^{\tmp}_{i'-1}$ from  each row $i'\in\{2,\dots,n+1\}\setminus\{i\}$,  and then extracting $\frac{1}{\beta^{\tmp}_{i-1}-r_{j'-1}}$ from each column $j'\in\{1,\dots,n+1\}$ we have  that
						\begin{align}\label{m3}
							\det[A_n]&={\frac{ \varrho^{3}_{i}}{\varrho^{4}_{i}}}\det[C^{2}],
							\end{align}
							with $\varrho^{3}_{i}\eqdef \prod_{\substack{k=2\\
											k\neq i}}^{n+1}[\beta^{\tmp}_{i-1}-\beta^{\tmp}_{k-1}]$,   $ \varrho^{4}_{i}\eqdef \prod_{k=1}^{n+1}[\beta^{\tmp}_{i-1}-r_{k-1}],$ 
							and $C^{2}=(c^{2}_{i',j'})$ where its entries are defined as
							\begin{equation*}
								c^{2}_{i',j'}=
								\begin{cases}
									-r_{j'-1}&\text{if}\  i'=1\ \text{and}\ j'\in\{1,\dots,n+1\},\\
									\beta^{\tmp}_{i-1}&\text{if}\ i'=i\ \text{and}\ j'\in\{1,\dots,n+1\},\\
									\frac{r_{j'-1}}{\beta^{\tmp}_{i'-1}-r_{j'-1}} &\text{if}\ i'\in\{2,\dots,n+1\}\setminus\{i\}\ \text{and}\ j'\in\{1,\dots,n+1\}.
								\end{cases}
							\end{equation*}
							By subtracting the column $j$ of $C^{2}$ from each of its  other  columns $j'\in\{1,\dots,n+1\}\setminus\{j\}$,  we have a matrix $C^{3}=(c^{3}_{i',j'})$ such that
							\begin{align*}
								&c^{3}_{i',j'}=
								\begin{cases}
									r_{j-1}-r_{j'-1}&\text{if}\  i'=1\ \text{and}\ j'\in\{1,\dots,n+1\}\setminus\{j\},\\
									-r_{j-1}&\text{if}\  i'=1\ \text{and}\ j'=j,\\
									\frac{r_{j-1}}{\beta^{\tmp}_{i'-1}-r_{j-1}} &\text{if}\ i'\in\{2,\dots,n+1\}\setminus\{i\}\ \text{and}\ j'=j\\
									0&\text{if}\ i'=i\ \text{and}\ j'\in\{1,\dots,n+1\}\setminus\{j\},\\
									\beta^{\tmp}_{i-1}&\text{if}\ i'=i\ \text{and}\ j'=j,\\
									\frac{-\beta^{\tmp}_{i'-1}[r_{j-1}-r_{j'-1}]}{[\beta^{\tmp}_{i'-1}-r_{j'-1}][\beta^{\tmp}_{i'-1}-r_{j-1}]} &\text{if}\ i'\in\{2,\dots,n+1\}\setminus\{i\}\\
									& \text{and}\ j'\in\{1,\dots,n+1\}\setminus\{j\}.
								\end{cases}
							\end{align*}
							Afterward, by extracting first $r_{j-1}-r_{j'-1}$ from  each column $j'\in\{2,\dots,n+1\}\setminus\{j\}$ {of $C^{3}$},  and then extracting $\frac{-1}{\beta^{\tmp}_{i'-1}-r_{j-1}}$ from each row $i'\in\{2,\dots,n+1\}\setminus\{i\}$ we rewrite \eqref{m3} as follows
							\begin{align*}
								\det[A_n]
								&={[-1]^{2n-1}\frac{ \varrho^{3}_{i}\varrho^{1}_{j}}{ \varrho^{4}_{i}\varrho^{5}_{i,j}}}\det\left(\begin{matrix}
									1 & \cdots  & 1 & -r_{j-1} &1 &\cdots  & 1\\
									\frac{\beta^{\tmp}_1}{\beta^{\tmp}_1-r_{0}} & \cdots& 	\frac{\beta^{\tmp}_1}{\beta^{\tmp}_1-r_{j-2}} & 	-r_{j-1} & 	\frac{\beta^{\tmp}_1}{\beta^{\tmp}_1-r_{j}} & \cdots & 	\frac{\beta^{\tmp}_1}{\beta^{\tmp}_1-r_{n}} \\
									\vdots & \ddots & \vdots & \vdots &\vdots & \ddots & \vdots \\
									\frac{\beta^{\tmp}_{i-2}}{\beta^{\tmp}_{i-2}-r_{0}} & \cdots& 	\frac{\beta^{\tmp}_{i-2}}{\beta^{\tmp}_{i-2}-r_{j-2}} & 	-r_{j-1} & 	\frac{\beta^{\tmp}_{i-2}}{\beta^{\tmp}_{i-2}-r_{j}} & \cdots & 	\frac{\beta^{\tmp}_{i-2}}{\beta^{\tmp}_{i-2}-r_{n}} \\
									0& \cdots& 	0 & 	\beta^{\tmp}_{i-1} & 	0& \cdots & 	0 \\
									\frac{\beta^{\tmp}_{i}}{\beta^{\tmp}_{i}-r_{0}} & \cdots& 	\frac{\beta^{\tmp}_{i}}{\beta^{\tmp}_{i}-r_{j-2}} & 	-r_{j-1} & 	\frac{\beta^{\tmp}_{i}}{\beta^{\tmp}_{i}-r_{j}} & \cdots & 	\frac{\beta^{\tmp}_{i}}{\beta^{\tmp}_{i}-r_{n}} \\
									\vdots & \ddots & \vdots & \vdots &\vdots & \ddots & \vdots \\
									\frac{\beta^{\tmp}_n}{\beta^{\tmp}_n-r_{0}} & \cdots& 	\frac{\beta^{\tmp}_n}{\beta^{\tmp}_n-r_{j-2}} & -r_{j-1} & 	\frac{\beta^{\tmp}_n}{\beta^{\tmp}_n-r_{j}} & \cdots & 	\frac{\beta^{\tmp}_n}{\beta^{\tmp}_n-r_{n}} 
								\end{matrix}\right)\notag\\
								&=[-1]^{2n-1}\beta^{\tmp}_{i-1}\frac{ \varrho^{3}_{i}\varrho^{1}_{j}}{ \varrho^{4}_{i}\varrho^{5}_{i,j}}\Cof_{i,j}[ A_{n}]=[-1]^{2n-1}\beta^{\tmp}_{i-1}[\beta^{\tmp}_{i-1}-r_{j-1}]\frac{ \varrho^{3}_{i}\varrho^{1}_{j}}{\varrho^{4}_{i}\varrho^{2}_{j}}\Cof_{i,j}[ A_{n}],		
							\end{align*}
							where ${\varrho^{5}_{i,j}\eqdef \prod_{\substack{k=2\\
											k\neq i}}^{n+1}[\beta^{\tmp}_{k-1}-r_{j-1}]}=\dfrac{\varrho^{2}_{j}}{\beta^{\tmp}_{i-1}-r_{j-1}}$. Notice that  $\varrho^{3}_{i}=[-1]^{n-i+1}\underbar{$\theta$}_{i}\,\bar{\theta}^{n}_{i}$, and $\varrho^{4}_{i}=[-1]^{n-i+2} \underbar{$\nu$}_{i}\,\bar{\nu}^{n}_{i}$,
							where $\underbar{$\theta$}_{i}$, $\bar{\theta}^{n}_{i}$,  $\underbar{$\nu$}_{i}$, and $\bar{\nu}^{n}_{i}$ are given by \eqref{co1}. From here and using \eqref{cof1} and \eqref{m2}, we get that $\det[A_{n}]
								=\beta^{\tmp}_{i-1}[\beta^{\tmp}_{i-1}-r_{j-1}]\frac{\underbar{$\theta$}_{i}\,\bar{\theta}^{n}_{i}\det[A_{n}]\Cof_{i,j}[ A_{n}]}{ \underbar{$\nu$}_{i}\,\bar{\nu}^{n}_{i}\Cof_{1,j}[A_{n}]}$. From this point, it becomes evident that \eqref{Cof2} is valid. 
								
								 {Let us now verify the validity of \eqref{Rn}. By \eqref{Cof2}, we have  
								\begin{equation}\label{m4}
									\sum_{i=2}^{n+1}\frac{\mathcal{M}^{n}_{i-1}}{\beta^{\tmp}_{i-1}-r_{j-1}}=\frac{1}{\Cof_{1,j}[A_{n}]}\sum_{i=1}^{n+1}\Cof_{i,j}[A_{n}]-1.
								\end{equation} 
								Using \eqref{Cof2} and the fact that  $\sum_{i=1}^{n+1}\Cof_{i,j}[A_{n}]=\det[A^{j}_{n}[I^{\trans}_{n+1}]]$, it follows immediately that \eqref{Rn} holds. }
						
							 {
								To conclude, we  prove by induction that \eqref{Cof5} is true.  For $n=1$,  \eqref{Cof5} is verified by direct calculation using \eqref{Rn}, \eqref{Cof3} and \eqref{Cof3.1}. Suppose that \eqref{Cof5} holds  for $n$. Using \eqref{m4} and taking into account the following identity, which can be verified similarly that \eqref{m4}, 
								\begin{equation*}
									1+\sum_{i=2}^{n+1}\frac{\mathcal{M}^{n}_{i-1}}{\beta^{\tmp}_{i-1}}=\frac{\mathcal{R}^{n}}{r_{j-1}}-\frac{r_{j-1}}{\Cof_{1,j}[A_{n}]}\sum_{i=2}^{n+1}\frac{\Cof_{i,j}[A_{n}]}{\beta^{\tmp}_{i-1}},
								\end{equation*} 
								it follows that
								\begin{align*}
									\frac{1}{\mathcal{R}^{n+1}}\left[1+\sum_{i=2}^{n+2}\frac{\mathcal{M}_{i-1}^{n+1}}{\beta_{i-1}^{\tmp}}\right]=&\frac{1}{r_{n+1}}-\frac{\sum_{i=2}^{n+2}\frac{\Cof_{i,n+2}[A_{n+1}]}{\beta^{\tmp}_{i-1}}}{\sum_{i=1}^{n+2}\Cof_{i, n+2}[A_{n+1}]}\notag\\
									&=\frac{1}{r_{n+1}}-\frac{1}{\beta_{n+1}^{\tmp}}\notag\\
									&\quad+\frac{\Cof_{1, n+2}[A_{n+1}]\bigg[1+\sum_{i=2}^{n+1}\bigg[\frac{\beta_{i-1}^{\tmp}-\beta_{n+1}^{\tmp}}{\beta_{i-1}^{\tmp}}\bigg]\frac{\Cof_{i,n+2}[A_{n+1}]}{\Cof_{1, n+2}[A_{n+1}]}\bigg]}{\beta_{n+1}^{\tmp}\sum_{i=1}^{n+2}\Cof_{i, n+2}[A_{n+1}]}.
									\end{align*}
									The last equality is obtained by adding and subtracting $1/\beta_{n+1}^{\tmp}$. Then, from here, by \eqref{Cof3}, \eqref{co1} and  \eqref{Cof2}, and noting that $\mathcal{R}^{n+1}=\mathcal{R}^{n}r_{n+1}/\beta_{n+1}^{\tmp}$, it can be verified that
									\begin{align*}
									\frac{1}{\mathcal{R}^{n+1}}\left[1+\sum_{i=2}^{n+2}\frac{\mathcal{M}_{i-1}^{n+1}}{\beta_{i-1}^{\tmp}}\right]&=\frac{1}{r_{n+1}}-\frac{1}{\beta_{n+1}^{\tmp}}+\frac{1}{\mathcal{R}^{n}}\left[1+\sum_{i=2}^{n+1}\Bigg[\frac{\beta_{i-1}^{\tmp}-\beta_{n+1}^{\tmp}}{\beta_{i}^{\tmp}}\Bigg]\frac{\Cof_{i,n+2}A_{n+1}}{\Cof_{1,n+2}A_{n+1}}\right],\notag\\
									&=\frac{1}{r_{n+1}}-\frac{1}{\beta_{n+1}^{\tmp}}+\frac{1}{\mathcal{R}^{n}}\left[1+\sum_{i=2}^{n+1}\frac{\mathcal{M}_{i-1}^{n}}{\beta_{i-1}^{\tmp}}\right].
								\end{align*}
								Hence, by the induction hypothesis, we conclude that\eqref{Cof5}holds for $n+1$.
							}\qed
						\end{proof}

						\noindent {\bf Acknowledgements} The authors thank the anonymous reviewer for their constructive comments and suggestions that substantially improved this paper.
						
						J. S. Mora Rodr\'iguez acknowledges the scholarship from the Faculty of Sciences, Universidad Nacional de Colombia, that supported his PhD studies.
						
				H. A. Moreno-Franco acknowledges the support received within the framework of the Basic Research Program at HSE University. \\
						
				\noindent 	{\bf Funding}  The research of J. Garz\'on was partially supported by the Project HERMES 58557.

						\section*{Declarations}

						\noindent	{\bf Conflict of interest} The authors have no competing interests to declare that are relevant to the content of this article.

\end{document}